\documentclass[11.5pt]{amsart}

\usepackage{amsmath}
\usepackage{amssymb}
\usepackage{graphicx}
\usepackage{xcolor}
\usepackage{pifont}
\usepackage{epsfig}
\usepackage{verbatim}
\usepackage[english]{babel}
\usepackage{fancyhdr, blindtext}
\usepackage{float}
\usepackage{subfig}

\usepackage{graphics}
\usepackage{graphicx,epsfig}
\usepackage{epsfig}
\usepackage{graphicx}

\usepackage{a4wide}
\usepackage{amsmath}
\usepackage{amsfonts}

\usepackage{enumerate}
\usepackage{authblk}
\usepackage{float}
\usepackage{subfig}
\usepackage{amsaddr}

\newcommand\bes{\begin{eqnarray}}
\newcommand\ees{\end{eqnarray}}
\newcommand\bess{\begin{eqnarray*}}
\newcommand\eess{\end{eqnarray*}}

\newtheorem{theorem}{Theorem}[section]
\newtheorem{lemma}{Lemma}[section]
\newtheorem{proposition}{Proposition}[section]
\newtheorem{remark}{Remark}[section]

\begin{document}
\title[Bistable dynamics and early warning signals]
{\bf Analysis of long transients and detection of early warning signals of extinction in a class of predator-prey models exhibiting bistable behavior.}

\author{S. Sadhu$^*$ }
\address{Department of Mathematics,
Georgia College \& State University, GA 31061, USA.}
\email{susmita.sadhu@gcsu.edu$^*$ }
 \author{S. Chakraborty Thakur$^+$}
\address{Department of Physics, Auburn University, AL 36849, USA.}
\email{szc0199@auburn.edu$^+$ }


\thispagestyle{empty}


\begin{abstract}
In this paper, we develop a method of analyzing long transient dynamics in a class of predator-prey models with two species of predators competing explicitly for their common prey, where the prey evolves on a faster timescale than the predators. In a parameter regime near a {\em{singular zero-Hopf bifurcation}} of the coexistence equilibrium state, we assume that the system under study exhibits bistability between a  periodic attractor that bifurcates from the singular Hopf point and another attractor, which could be a  periodic attractor or a point attractor, such that the invariant manifolds of the coexistence equilibrium point play central roles in organizing the dynamics. To find whether a solution that starts in a vicinity of the coexistence equilibrium approaches the periodic attractor or the other attractor, we reduce the equations to a suitable normal form, and examine the basin boundary near the singular Hopf point. A key component of our study includes an analysis of the long transient dynamics, characterized by their rapid oscillations with a slow variation in amplitude, by applying a moving average technique. We obtain a set of necessary and sufficient conditions on the initial values of a solution near the coexistence equilibrium to determine whether it lies in the basin of attraction of the periodic attractor.  As a result of our analysis, we devise a method of identifying early warning signals, significantly in advance, of a future crisis that  could lead to extinction of one of the predators. The analysis is applied to the predator-prey model considered in [\emph{Discrete and Continuous Dynamical Systems - B} 2021, 26(10), pp. 5251-5279] and we find that our theory is in good agreement with the numerical simulations carried out for this model.
\end{abstract}

\maketitle

\vspace{0.15in}



\section{Introduction}
Identifying early warning signals for anticipating population collapses or finding early indicators of recovery of a threatened or an endangered population is a major focus of research in nature conservation and ecosystem management. Often the collapses can be catastrophic as they may not be reverted, causing  huge  changes in ecosystem structure and function while leading to  extinction of species and substantial loss of biodiversity \cite{Hastings1, Hastings2,  morozovetal}. Examples of catastrophic collapses include the  extinction of Steller's sea cow \emph{Hydrodamalis gigas} within just 30 years of its discovery  in the late 18th century, the dramatic extinction of passenger pigeon \emph{Estopistes migratorius} in the early 20th century, the collapse of the Atlantic northwest cod fishery in the early 1990s, the extinction of sea urchins {\emph{Paracentrotus lividus}} in South Basin of Lough Hyne in Ireland in the early 2000s, and sharp decline of coral cover on the Great Barrier Reef. While collapse of a population is of primary concern for conservationists, early indicators of recovery of a system can play a crucial role in assessing the effectiveness of intervention strategies and resource management decisions \cite{clements}. Consequently, changes in abundance or fitness-related traits are usually monitored to find indicators of impending shifts in population dynamics of threatened species.
A common trend observed in the population of many endangered species is an abrupt transition from a seemingly steady state, persisting over many generations, to an unsustainable state that may lead to its extinction. The persistence of long transient phases in population ecology has been also documented in many empirical observations and theoretical models of various species \cite{Hastings1, Hastings2}, which then brings forth the subject of developing suitable mathematical methods to analyze transient dynamics in ecological models.

Abrupt shifts in ecosystems or ``regime shifts" have been attributed to bifurcations in the underlying system driven by slow changes in environmental parameters, random fluctuations or external forcing, the drifting rate of time-dependent parameters, or a combination of some of these factors  \cite{ashwin, SC, Scheffer, ZKE}. Early warning signals of regime shifts have usually been related to the phenomenon of ``critical slowing down" \cite{carpenter} that arises in the vicinity of local bifurcations, and 
 the commonly known measures for detecting them include examination of changes in trends of statistical measures such as variance, autocorrelation, and skewness in time series \cite{boettigera2, SC, Scheffer}. However, there is also a growing recognition of the leading role of long transient dynamics in explaining regime shifts and developing strategies for efficient ecological forecasting, since regime shifts or critical transitions between alternative states can be viewed as phenomena induced by transients-related  mechanisms \cite{Hastings2, morozovetal}. In this paper, we take the latter outlook to study regime shifts. Moreover, rather than considering statistical indicators of critical transitions, we take an alternative approach for detecting early warning signals by identifying and monitoring a suitable combination of the state variables that resulted from analyzing the mechanism driving the long transients in the system. Though the importance of long transients has been recognized in mathematical ecology \cite{Hastings1}, analytical techniques for studying transient dynamics in relevant timescales in three or higher-dimensional models is still at its mathematical infancy \cite{heggerudetal, sadhutrans}. With this spirit, in this paper, we adopt a dynamical systems approach to analyze long transients and apply the analysis to predict an impending transition in population dynamics of a  class of three-species predator-prey models exhibiting bistability between a limit cycle and a boundary equilibrium state in a parameter regime near a {\em{singular zero-Hopf bifurcation}} \cite{BE, BB, DGKKOW, G, GH, K}. The goal is to predict the long-term behaviors of solutions exhibiting similar oscillatory dynamics as  transients and devise a method of  identifying an early warning signal significantly in advance of an abrupt transition that could lead to an extinction. %

The work in this paper is inspired by the dynamics of the model studied in \cite{Sadhunew} which reads as
\begin{eqnarray}\label{maineq} \left\{ 
\begin{array}{ll}\label{1}
      \frac{dR}{dT} &= rR\left(1-\frac{R}{K}\right)-\frac{p_1RC_1}{H_1+R}-\frac{p_2RC_2}{H_2+R}\\
        \frac{dC_1}{dT} &= \frac{q_1p_1RC_1}{H_1+R}-m_1C_1 -a_{12}C_1C_2\\
         \frac{dC_2}{dT} &= \frac{q_2p_2RC_2}{H_2+R}-m_2C_2 -a_{21}C_1C_2 -a_{22}C_2^2,
       \end{array} 
\right. \end{eqnarray}
where $R$ represents the population density of the prey; $C_1$, $C_2$ represent the densities of the two species of  competing predators;  $r>0$ and $K>0$ are the intrinsic growth rate and the carrying capacity of the prey; $p_1>0$ is the maximum per-capita predation rate of $C_1$, $H_1>0$ is the semi-saturation constant which represents the prey density at which $C_1$ reaches half of its maximum predation rate ($p_1/2$),  $q_1>0$ is the birth-to-consumption ratio of $C_1$, $m_1>0$ is the  per-capita natural death rate of $C_1$, and $a_{12}>0$ is the rate of adverse effect of $C_2$ on $C_1$. The other parameters $p_2, q_2, m_2, H_2, a_{21}$ are defined analogously and $a_{22}>0$ represents the intraspecific competition within $C_2$. It is assumed that $R$ evolves on a faster timescale than $C_1$ and $C_2$, and $C_2$ is a stronger competitor than $C_1$. Depending on the relative strengths of $a_{22}$ and $a_{21}$, $C_2$ may either outcompete $C_1$, or the three species may coexist in an equilibrium or an oscillatory state in the form of small-amplitude oscillations, \emph{mixed-mode oscillations} \cite{BKW, DGKKOW, SCT}, or relaxation oscillations. 

System (\ref{maineq}) in its non-dimensional form (see  \cite{Sadhunew} for the details) can be written as
\begin{eqnarray}\label{nondim3}    \left\{
\begin{array}{ll} \zeta\dot{x}&= x\left(1-x-\frac{y}{\beta_1+x}-\frac{z}{\beta_2+x}\right):=x\phi(x,y,z)\\
 \dot{y}&=y\left(\frac{x}{\beta_1+x}-d_1-  \alpha_{12} z \right):=y\chi(x,y,z)\\
   \dot{z} &=   z\left(\frac{x}{\beta_2+x}-d_2 -\alpha_{21} y-hz \right):=z\psi(x,y,z),
       \end{array} 
\right. \end{eqnarray}
where the overdots  denote differentiation with respect to the slow time $s = \zeta r T$, where $\zeta$ measures the ratio of the growth rates of the predators to the prey and satisfies $0<\zeta \ll 1$. The nontrivial $x$, $y$, and $z$-nullcline in (\ref{nondim3}) are denoted by  $\phi=0$, $\chi=0$, and $\psi=0$ respectively. The dimensionless parameters $\beta_1$, $d_1$ and $\alpha_{12}$ respectively represent the predation efficiency \cite{BD1}, rescaled mortality rate and rescaled interspecific competition coefficient of $y$. The parameters $\beta_2$, $d_2$, and $\alpha_{21}$ are analogously defined and $h$ represents the rescaled intraspecific competition coefficient within $z$.

Interestingly, it was observed  in \cite{Sadhunew} that for suitable parameter values near a subcritical singular Hopf bifurcation,  system (\ref{nondim3})  exhibits bistability between a small-amplitude limit cycle and a boundary equilibrium state which corresponds to the  extinction of the weaker predator (figure 21 in \cite{Sadhunew}).  It turns out that in this regime, the dynamics are organized by the invariant manifolds of a nearby saddle-focus coexistence equilibrium point, allowing a solution that starts near the  coexistence equilibrium point to approach it along its two-dimensional stable manifold and then escape along its one-dimensional unstable manifold either towards the boundary equilibrium state or towards a nearby limit cycle (see figure \ref{initial_transients}).  In both cases, the local dynamics near the coexistence equilibrium appear very similar and are characterized by high-frequency oscillations with slowly varying amplitude along the invariant manifolds of the equilibrium. The similarity in the initial trends makes identification of  early warning signals of a population collapse extremely challenging (see figures \ref{timeseries_all} and \ref{initial_transients}), which is one of the subjects of investigation in this work. 

   \begin{figure}[htbp]    
  \centering 
  \subfloat[IC: $(0.2785, 0.1181, 0.4164)$]{\label{timeseries_all_a}\includegraphics[width=6.6cm]{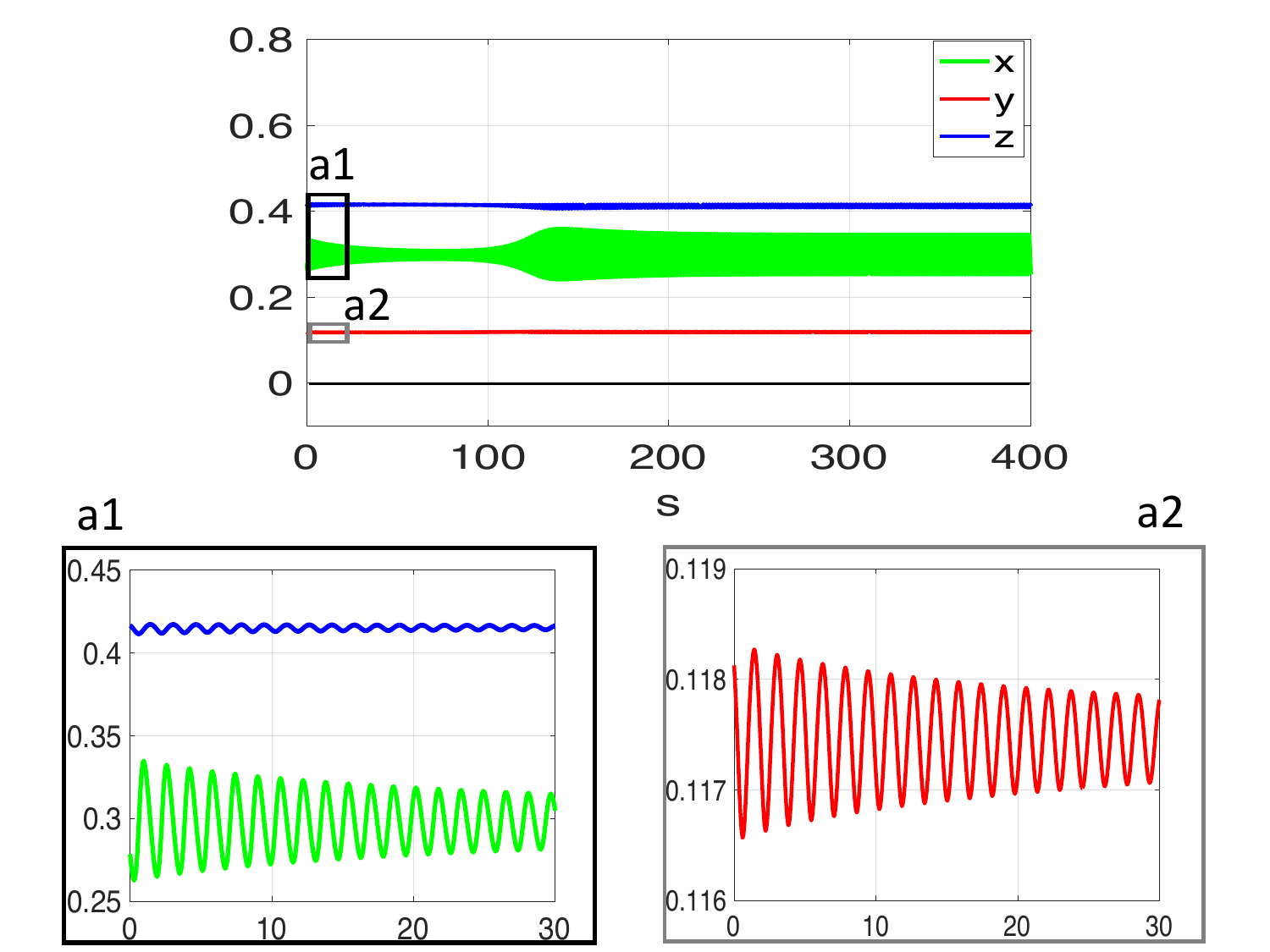}}
    \subfloat[IC: $(0.2831,0.1184, 0.417)$]{\label{timeseries_all_b}\includegraphics[width=6.6cm]{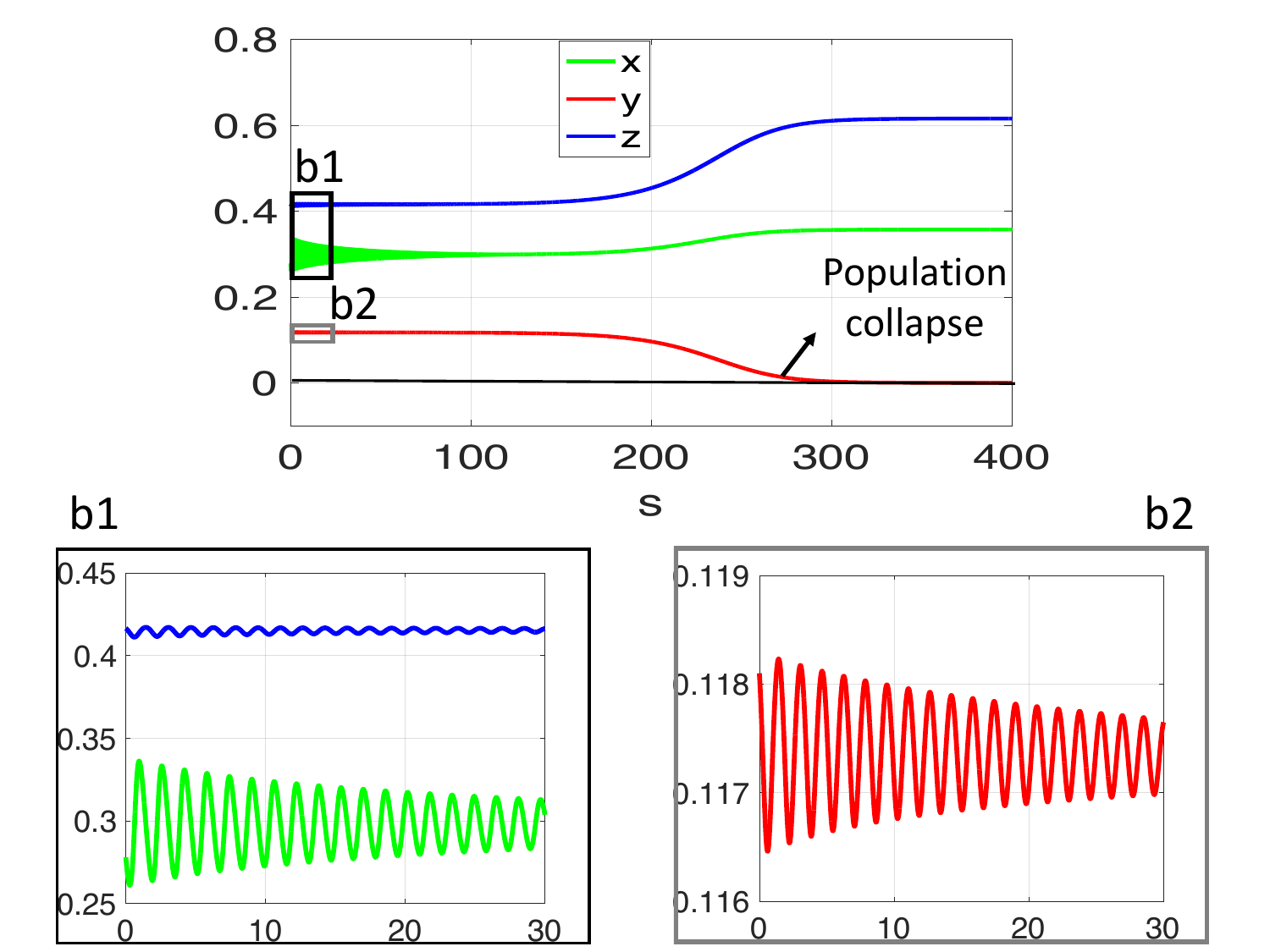}}
 \caption{An illustration of bistable behavior exhibited by system (\ref{nondim3}) for $h=0.2649$ and other parameter values as in (\ref{parvalues})). Both solutions exhibit similar patterns of oscillations initially (see insets a1-a2 and b1-b2) but  have different asymptotic behaviors. IC: initial conditions.
 }
  \label{timeseries_all}
\end{figure} 

 \begin{figure}[h!]     
  \centering 
         \subfloat[IC: $(0.2785, 0.1181, 0.4164)$]{\includegraphics[width=6.5cm]{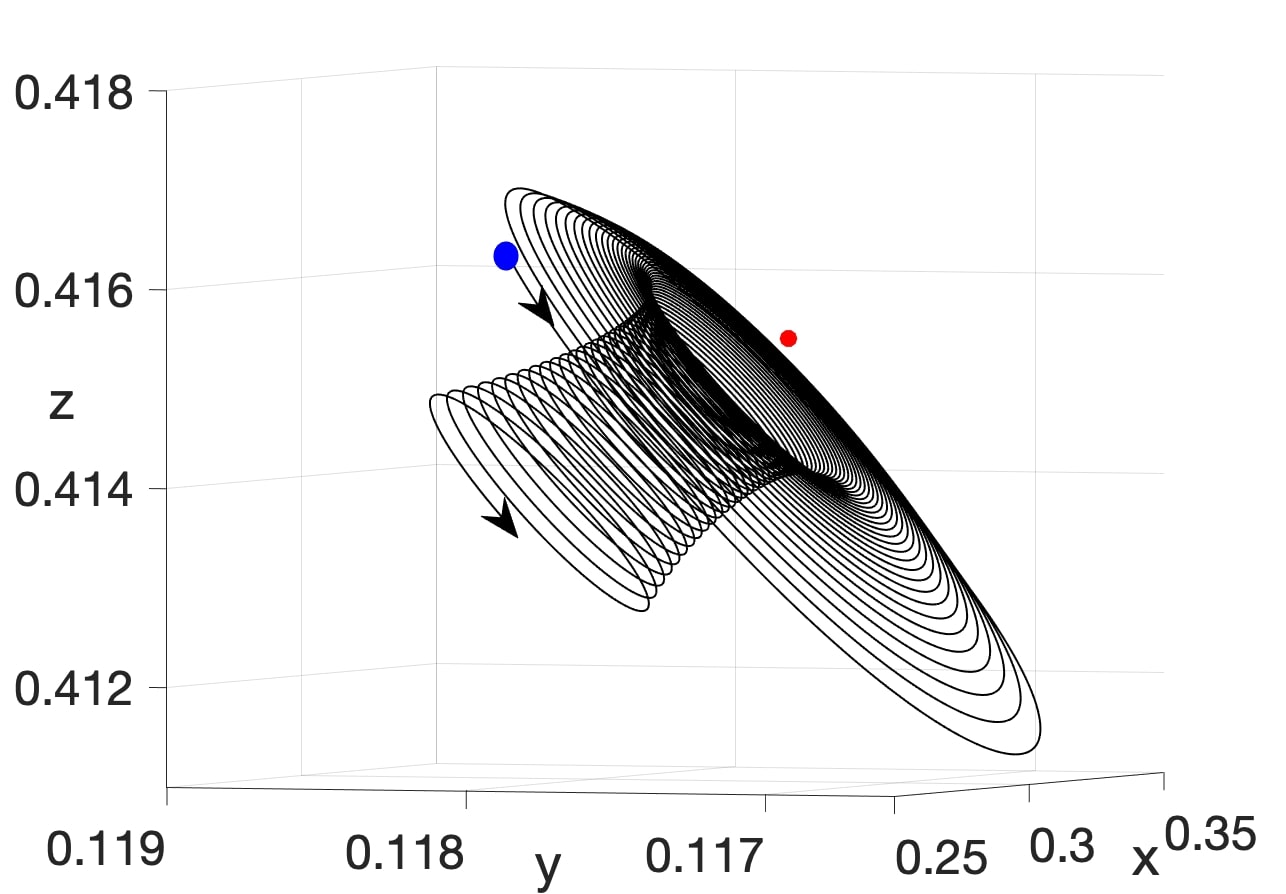}}\quad 
      \subfloat[IC: $(0.2831,0.1184, 0.417)$]{\includegraphics[width=6.5cm]{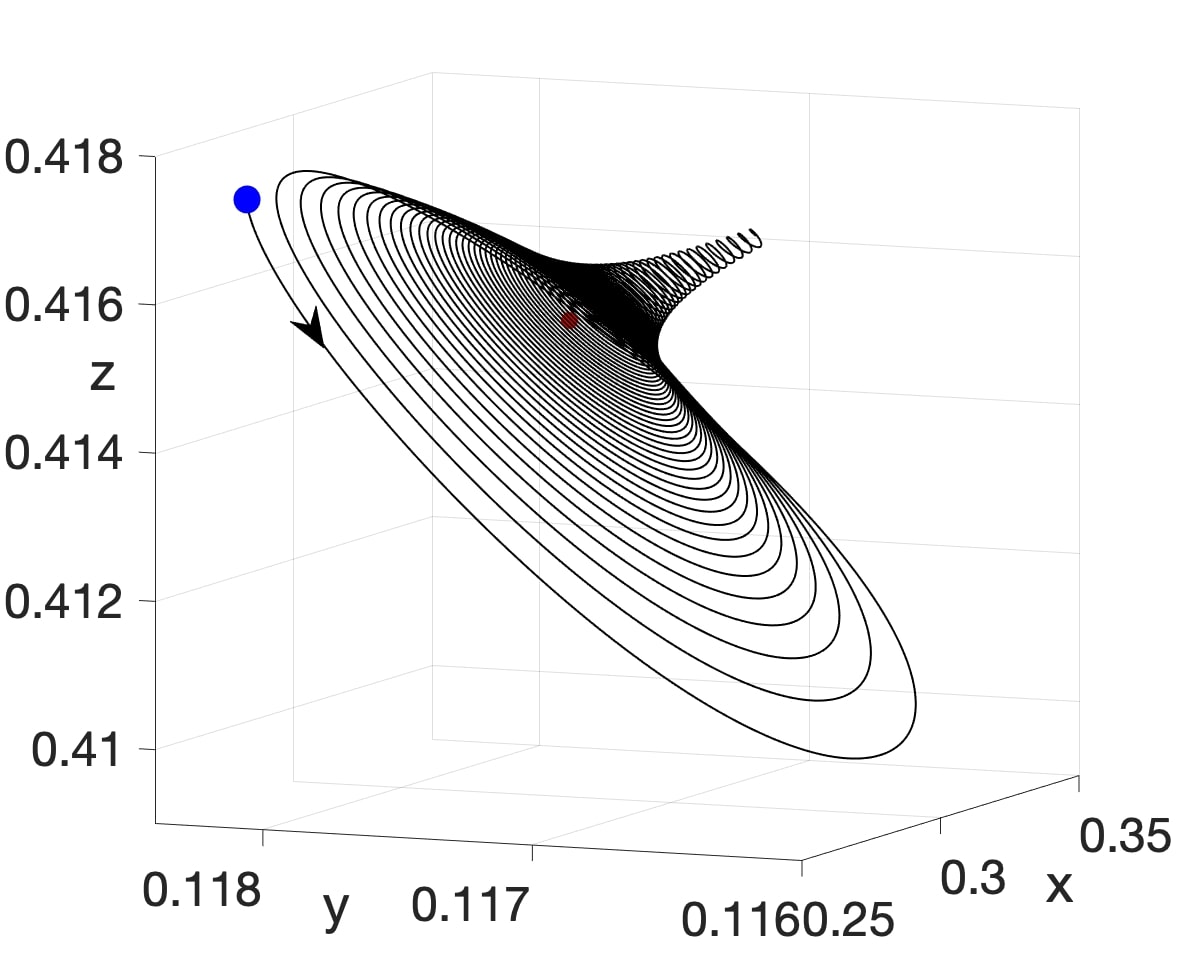}}
 \caption{The phase portraits of the transient dynamics  corresponding to the time series in figure \ref{timeseries_all} for $ s \in [0, 100]$.  The solutions spend a significant amount of time near the coexistence equilibrium point shown by the red dot before approaching their asymptotic states. The initial conditions are shown by the blue dots. (A) The solution asymptotically approaches a periodic attractor. (B) The solution asymptotically approaches a boundary equilibrium point.}
 \label{initial_transients}
\end{figure}

Another aspect that is considered in this study pertains to an analysis of the mechanism underlying long transients in systems exhibiting such bistable dynamics. Some of the commonly known mechanisms responsible for the emergence of long transients in ecological models include the proximity of a system to a ghost attractor, slow passage through a saddle known as a ``saddle crawl-by", presence of slow manifolds in systems with timescale separation, or effects of stochasticity or time delays \cite{Hastings2, morozovetal}. However, in this paper, we reveal another mechanism that can lead to long transients, namely, a slow passage through a saddle-focus equilibrium that lies near a zero-Hopf bifurcation in the parameter space. It turns out that the proximity of the system to a zero-Hopf bifurcation results in a weak contraction and a weak expansion along the stable and unstable manifolds of the saddle-focus equilibrium, allowing a solution to spend a significant amount of time near the invariant manifolds as it passes through the equilibrium,  thus generating long transients as shown in figure \ref{initial_transients}. To the best of our knowledge, such a mechanism underlying long transients in three or higher-dimensional models has not been explored yet.

To analyze the bistable behavior and the transient dynamics, the general class of models with two intrinsic timescales, which includes system (\ref{maineq}) as a special case,  is reduced to a  topologically equivalent class, referred to as the normal form \cite{BB}, which is valid near the singular Hopf point where certain conditions on the derivatives hold. The normal form consists of two-fast variables and one-slow variable, where the slow variable represents the slowly changing dynamics along the axial direction corresponding to the unstable set of the saddle-focus equilibrium, while the fast variables feature rapid oscillations with slow variation in amplitude that occur along the stable set of the equilibrium.  Consequently, the transformed coordinates system associated with the normal form naturally captures the ``near cylindrical symmetry" and the intrinsic nature of the  dynamics near the saddle-focus equilibrium. Exploiting the separation of timescales between the rapid oscillations and the slow variation in the amplitude, we apply the method of averaging \cite{GH, V} to partition the system into the sum of its time-averaged and fluctuating parts and consider the equation of the time evolution of the mean of the oscillations in the axial direction. We solve the associated ODE and express the solution in terms of the mean of the oscillations of the fast subsystem, and study its properties. The analysis is then used to find a set of sufficient conditions to distinguish between solutions that start in a vicinity of the saddle-focus point with contrasting long-term behaviors.

The normal form also reveals the underlying geometrical structure of the system and gives insight into the landscape of the basin boundary of the two attractors in the vicinity of the saddle-focus equilibrium. We explicitly find a region in the phase space, bounded by the surface of an elliptic paraboloid, that lies in the basin of attraction of the periodic attractor. The 3-D region, which we refer to as the ``funnel" encloses a branch of the unstable manifold of the saddle-focus equilibrium and separates orbits approaching the limit cycle from the point attractor. We prove that a solution in the vicinity of the stable manifold of the saddle-focus equilibrium must enter into this ``funnel" to asymptotically approach the periodic attractor. We find a set of necessary and sufficient conditions on the initial data corresponding to which a solution eventually enters the funnel as it escapes along the unstable manifold of the saddle-focus equilibrium. Consequently, we obtain a critical threshold on the slowly varying dynamics along the unstable manifold, such that if a solution goes below it, it must approach the boundary equilibrium state. Finally, the analysis is used to devise a method for finding early warning signals of the onset of a drastic change in the population of one of the species of predators. The results are in good agreement with the numerical simulations carried out for  system (\ref{maineq}).

The remainder of the paper is organized as follows. In Section \ref{sec:gen_formulation}, we present the general formulation of the class of equations and lay down the assumptions. We briefly discuss the geometric structure of system (\ref{maineq}) and perform few numerical investigations to gain insight into the dynamics.  In Section \ref{sec:normalform}, we reduce the general system to a topologically equivalent form, which is valid near the singular Hopf bifurcation and analyze the equivalent system. We obtain a set of sufficient conditions on the initial data in a neighborhood of the coexistence equilibrium to lie in the basin of attraction of the periodic attractor.  Using the analysis, we then devise a method of detecting an early warning signal of population collapse in Section \ref{sec:early_warning}. The results are supported by numerical simulations in Section \ref{sec:num_results}. Finally, we summarize our conclusion in Section \ref{sec:conclusion}.

\section{Mathematical Model}
 \label{sec:gen_formulation}

\subsection{General Formulation} The class of equations under consideration is of the form
\begin{eqnarray}\label{normal1}    
\left\{
\begin{array}{ll} \zeta \dot{x} &=f_1(x, y, z,h) :=x\phi(x,y,z,h) \\
    \dot{y}&=f_2(x, y, z,h) :=y\chi(x,y,z,h) \\
    \dot{z}&= f_3(x,y, z,h):= z\psi(x,y,z,h),
       \end{array} 
\right. 
\end{eqnarray} 
where $x, y, z$ rehresent the population densities of the prey and the two species of predators respectively, $p$ is a parameter in a compact subset of $\mathbb{R}$ and $\phi,\  \chi$, and $\psi$ are smooth functions having the properties described by (H1)-(H7) and (Q1)-(Q5) below. The overdots in (\ref{normal1}) denote differentiation with respect to the time variable $s$ and $0<\zeta \ll 1$ represents a singular parameter. We assume that the prey evolves on a faster timescale than the predators, a phenomenon commonly observed in many ecosystems \cite{BD1, DH, MR, MR1}. We make the following assumptions:\\

(H1)  $\phi(0,0,0,h)>0$, $\chi(0, 0,0, h)<0$, $\psi(0, 0, 0, h)<0$, $\phi(1,0,0,h)=0$, $\chi(1, 0, 0, h)>0$, $\psi(1, 0, 0,h)>0$ and $\phi_x(1, 0, 0, h) <0$.\\

(H2) $\phi_x(x, y, z, h)=0$ defines a smooth curve $\mathcal{F}$ on the surface $S:=\{(x, y, z) \in \mathbb{R}^3: \phi(x, y, z ,h)=0\}$ dividing $S$ into two smooth surfaces:
\[ S^a=\{(x, y, z)\in S : \phi_x(x,y, z, h)<0\},\   S^r=\{(x, y, z)\in S : \phi_x(x,y, z, h)>0\}.\]

Condition (H1) implies that the equilibrium points $(0, 0, 0)$ and $(1, 0, 0)$ are  saddles, where $(0, 0, 0)$ is attracting in the invariant $yz$-plane and repelling along the invariant $x$-axis while $(1, 0, 0)$ is attracting along the invariant $x$-axis and repelling along the $yz$ directions. Depending on  $h$, (\ref{normal1}) may admit other equilibria, some of which may lie on the $xy$-plane or the $xz$-plane. These equilibria will be referred to as the boundary equilibria and will be denoted by $E_{xy}$ and  $E_{xz} $ respectively. Rescaling $s$ by $\zeta$ and letting $t=s/\zeta$, system (\ref{normal1}) can be reformulated as \begin{eqnarray}\label{normal1_fast}    
\left\{
\begin{array}{ll} x' &=f_1(x, y, z,h) \\
y'&=\zeta f_2(x, y, z,h)  \\
z'&=\zeta  f_3(x,y, z,h),
       \end{array} 
\right. 
\end{eqnarray}
where the primes denote differentiation with respect to $t$. System (\ref{normal1_fast}) is referred to as the fast system, whereas (\ref{normal1}) as the slow system. 
The set of equilibria of (\ref{normal1_fast}) in its singular limit is called the {\emph{critical manifold}}  \cite{DGKKOW},  $\mathcal{M}: =\Pi\cup S$, where $\Pi$ is the invariant $yz$ plane and $S$ is the surface as defined in (H2).   
Hypothesis (H2) implies that the surface $S^a$ is normally attracting while $S^r$ is normally repelling with respect to the limiting fast system. The normal hyperbolicity of $S$ is lost via saddle-node bifurcations along the fold curve $\mathcal{F}$. 

We assume that there exists a point $(\bar{x}, \bar{y}, \bar{z}, \bar{h})$ such that the linearization of (\ref{normal1}) at this point has a pair of eigenvalues that approach infinity as $\zeta \to 0$ and the following conditions on $f_1$, $f_2$, $f_3$ and their derivatives hold at $(\bar{x}, \bar{y}, \bar{z}, \bar{h})$:

 (H3) $\bar{\phi}=0,\  \bar{\chi}=0,\  \bar{\psi}=0$.

 (H4) $\det J \neq 0$, where $J=  \begin{pmatrix}
(\bar{f}_1)_ x  & (\bar{f}_1)_y  &(\bar{f}_1)_z\\
(\bar{f}_2)_x & (\bar{f}_2)_y  & (\bar{f}_2)_z\\
(\bar{f}_3)_x & (\bar{f}_3)_y  & (\bar{f}_3)_z
 \end{pmatrix}$.

 (H5) $\bar{\phi_x}=0$,  $\bar{\phi}_{xx} \neq 0$.\\

 (H6)  $\begin{pmatrix}
 (\bar{f}_1)_y  & (\bar{f}_1)_z
 \end{pmatrix}\begin{pmatrix}
 (\bar{f}_2)_x \\
 (\bar{f}_3)_x 
 \end{pmatrix} <0.$

 (H7) $-\begin{pmatrix} (\bar{f}_1)_{xx} & (\bar{f}_1)_{xy} & (\bar{f}_1)_{xz}
\end{pmatrix} J^{-1}
 \begin{pmatrix}
(\bar{f}_1)_h  \\
(\bar{f}_2)_h\\
(\bar{f}_3)_h
 \end{pmatrix} +(\bar{f}_1)_{xh}\neq 0$.

Here the bars denote the values of the expressions evaluated at $(\bar{x}, \bar{y}, \bar{z}, \bar{h})$.
 Note that  (H3) indicates that $(\bar{x}, \bar{y}, \bar{z})$ is a coexistence equilibrium of system (\ref{normal1}) for $h=\bar{h}$. Hypothesis (H4) implies the existence of a smooth family of equilibria $(x_0(h), y_0(h), z_0(h))$ in a neighborhood of $\bar{h}$ via the implicit function theorem. 
 Hypothesis (H5) indicates that $(\bar{x}, \bar{y}, \bar{z})$ is a non-degenerate fold point.   Hypothesis (H6) implies that the linearization of system (\ref{normal1})  at the family of equilibria $E^*(h)=(x_0(h), y_0(h), z_0(h))$ admits a pair of eigenvalues with singular imaginary parts for sufficiently small $\zeta$ (see Prop. 1 in \cite{BB}). Finally (H7) implies that $\frac{d\sigma}{dh}\neq 0$ at $(\bar{x}, \bar{y}, \bar{z},\bar{h})$, where $\sigma$ is the real part of the pair of eigenvalues with singular imaginary parts of the linearization of system (\ref{normal1}) at the equilibrium.

The family of equilibria $(x_0(h), y_0(h), z_0(h))$ will be referred to as the coexistence equilibria and will be denoted by $E^*(h)$. The family undergoes a {\emph{singular Hopf bifurcation}} \cite{BB, DGKKOW} at $h_H=\bar{h}+O(\zeta)$ for  sufficiently small $\zeta>0$ (see Theorem \ref{normal}). 
Finally we assume that\\

 (Q1) The equilibrium  $E^*(h)=(x_0(h), y_0(h), z_0(h))$ switches from a saddle-focus with a two-dimensional stable and a one-dimensional unstable manifold to an unstable focus at a subcritical  Hopf bifurcation $h_H$, where a unique curve of saddle cycles, $\Gamma_h$, bifurcates from it. Furthermore, if $\lambda_{1, 2}(h)\in \mathbb{C}$ and $\lambda_3(h)>0$ are eigenvalues of the linearized system at $E^*(h)$ near $h_H$, then $\lambda_3(h) =O(|\Re(\lambda_{1,2})(h)|)$ for all $|h-h_H|=O(\zeta)$.\\
 
 (Q2)  If $Q_h$ be a Poincar\'e map associated with  $\Gamma_h$, then $Q_h$ undergoes a subcritical Neimark-Sacker bifurcation \cite{K} at an $O(\zeta)$ neighborhood of $h_H$, resulting into the stabilization of $\Gamma_h$.\\

 (Q3) There exists a family of asymptotically stable attractors $\tilde{\Gamma}(h)$ for all $|h-\bar{h}|=O(\zeta)$ such that $\Gamma(h)$ and $\tilde{\Gamma}(h)$  lie on opposite sides of the tangent plane, $\Sigma(h)$, to the local stable manifold of  $E^*(h)$. Furthermore, ${\bf{F}}\cdot {\bf{n}}>0$, where ${\bf{F}}$ is the vector field corresponding to system (\ref{normal1}) and ${\bf{n}}$ is the normal vector to $\Sigma$ pointing towards the lower-half space, $\Sigma^-$,  containing $\tilde{\Gamma}(h)$ (say).\\
 
  (Q4) The unstable manifold of the saddle-focus equilibrium $E^*(h)$  approaches the stable manifold of $\Gamma(h)$ in the upper-half space $\Sigma^+$, whereas it tends to the stable manifold of $\tilde{\Gamma}(h)$ in the lower-half space $\Sigma^-$ for all $|h-h_H|=O(\zeta)$.\\
  
  (Q5) There exists no other attractors other than $\Gamma(h)$ that lies in an immediate vicinity of $E^*(h)$.

  \begin{remark} \label{hypth} It follows from (Q3) that the lower half-space $\Sigma^-$ is positively invariant with respect to the flow generated by the vector field in (\ref{normal1}). Furthermore, (Q4) indicates that there exist maximal open sets $\mathcal{U}, \mathcal{V} \subset \Sigma^+$ around $E^*$ such that the flow from $\mathcal{U}$ and  $\mathcal{V}$  asymptotically approaches $\Gamma(h)$ and $\tilde{\Gamma}(h)$ respectively in the parameter regime where $\Gamma(h)$ and  $\tilde{\Gamma}(h)$ coexist as locally asymptotically stable attractors.
  \end{remark}

    \begin{figure}[h!]     
  \centering 
{\includegraphics[width=8.8cm]{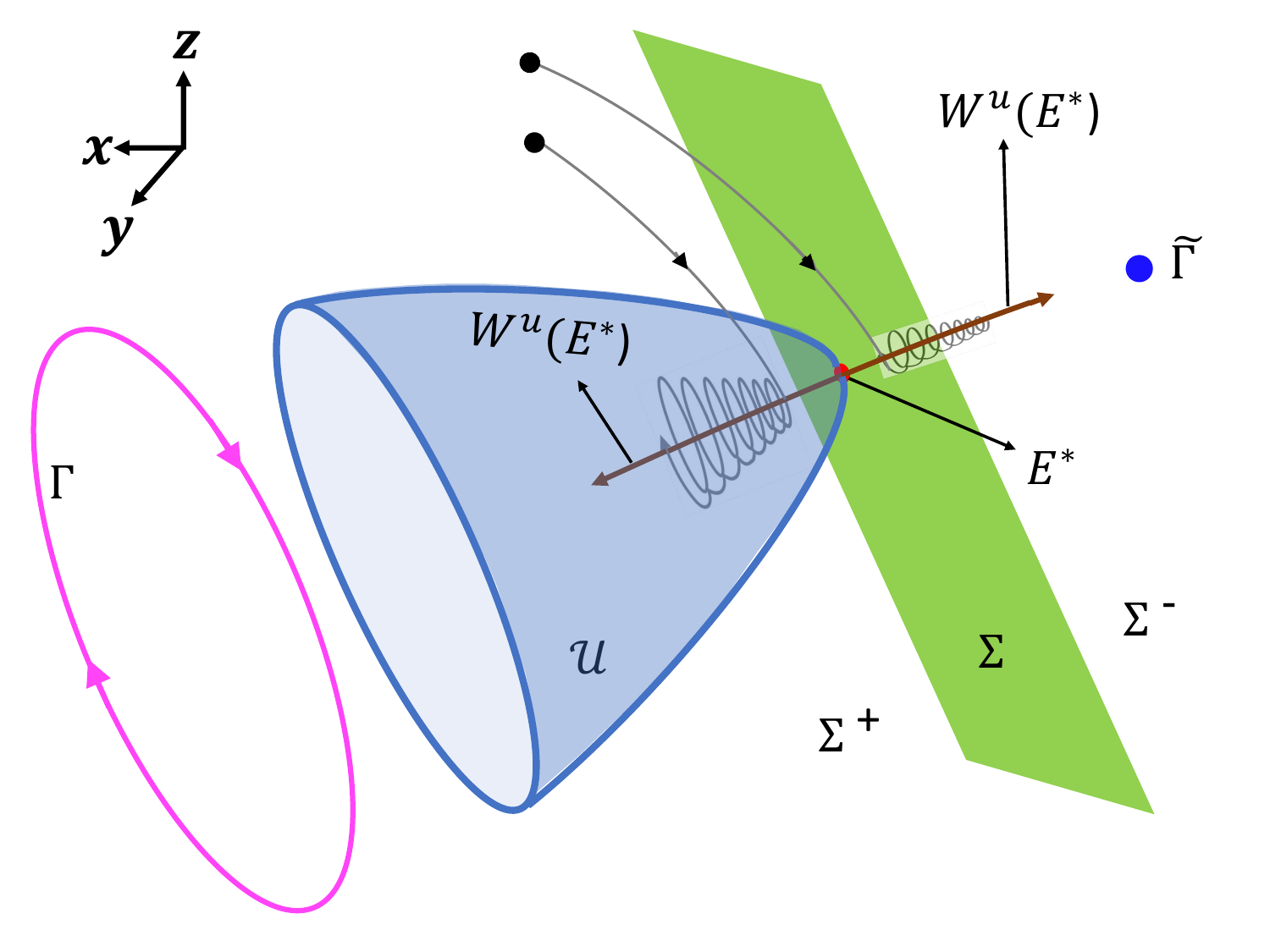}}
 \caption{A schematic representation of the geometrical configuration of system  (\ref{normal1}) near $E^*$. Shown are the open set $\mathcal{U}$, the tangent plane $\Sigma$ of $W^s(E^*)$, and two representative solutions starting in $\Sigma^+$ where one of them approaches $\Gamma$ and the other approaches $\tilde{\Gamma}$.}
 \label{geom_confg}
\end{figure}
 
Figure \ref{geom_confg} shows a schematic view of the dynamics in a small neighborhood of $E^*$ and the geometrical configuration that supports the bistable behavior in system  (\ref{normal1}).  The two attractors represented by $\Gamma$ and $\tilde{\Gamma}$ lie on the opposite sides of the tangent plane $\Sigma$. For two solutions starting in $\Sigma^+$ in a vicinity of $E^*$ (that may spend a significant amount of time near $E^*$ before getting repelled (c.f figure \ref{initial_transients} and figure \ref{normal_ratio_bdd_unbdd}), the one that enters into $\mathcal{U}$ approaches $\Gamma$, whereas the other that doesn't enter $\mathcal{U}$, crosses $\Sigma$ and approaches $\tilde{\Gamma}$.

\subsection{A predator-prey model}

An example of the class (\ref{normal1}) that satisfies hypotheses (H1)-(H7) along with (Q1)-(Q5) is system (\ref{nondim3}).  As in \cite{Sadhunew}, we will assume that   $0< d_1, \ d_2, \ \beta_1,\  \beta_2\  \alpha_{12}, \ \alpha_{21} <1$. It is easy to verify that system (\ref{nondim3}) satisfies the assumptions stated in (H1)-(H2).
In the singular limit, the {\emph{reduced flow}} corresponding to (\ref{nondim3}), restricted to the surface $S$, has singularities along the fold curve $\mathcal{F}$, referred to as the {\em{folded singularities}} or {\em{canard points}} \cite{DGKKOW}.  These singularities are analyzed by desingularizing the reduced flow (see \cite{Sadhunew}). 
 If an equilibrium of (\ref{nondim3}) and a folded singularity merge together and split again, interchanging their type and stability, then a {\emph{folded saddle-node bifurcation of type II}} (FSN II)  occurs \cite{DGKKOW}, followed by  a singular Hopf bifurcation in an $O(\zeta)$ neighborhood of  the FSN II bifurcation.

  \begin{figure}[h!]     
  \centering 
{\includegraphics[width=8.8cm]{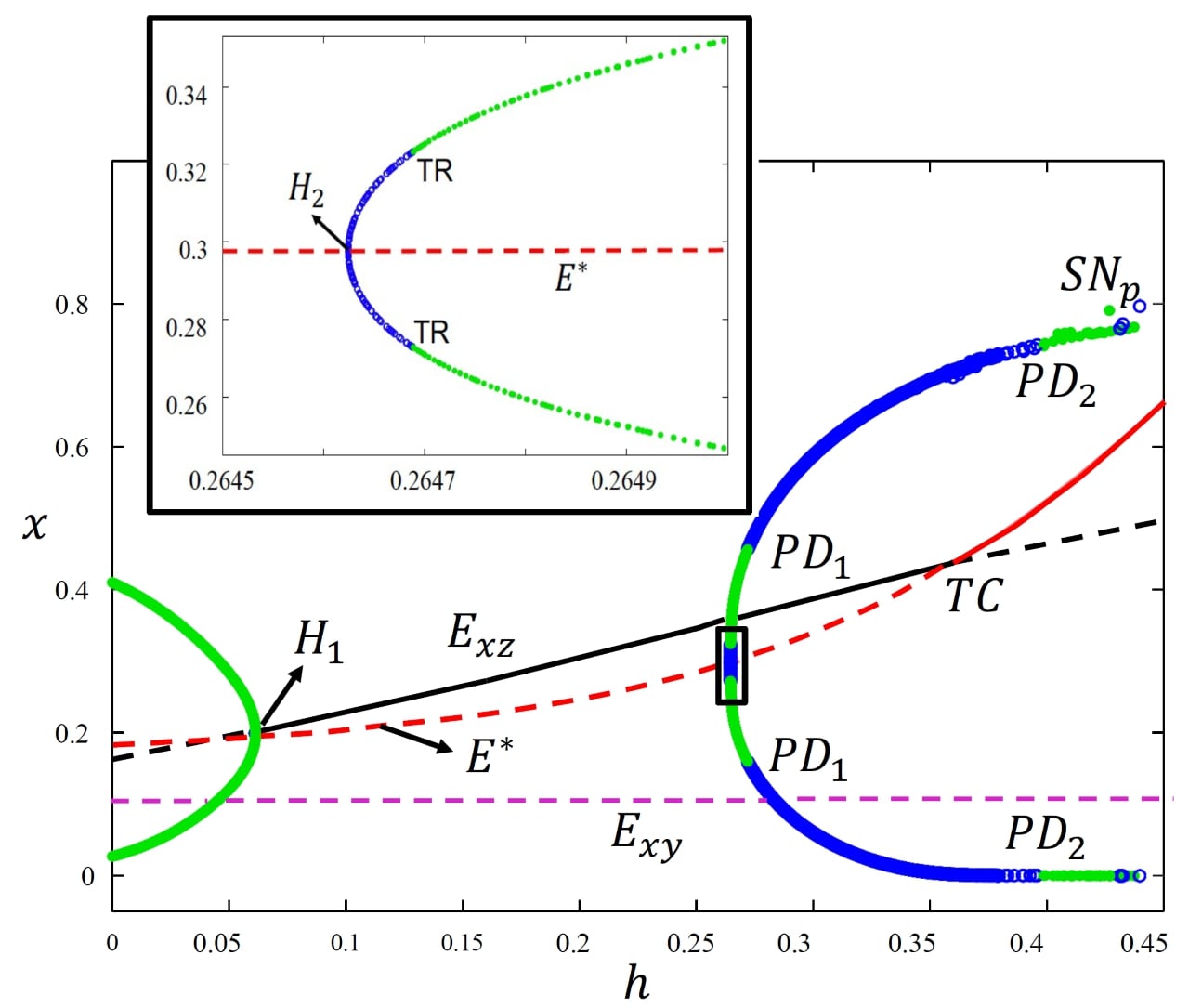}}
 \caption{One-parameter bifurcation diagram of system (\ref{nondim3}) as $h$ is varied. Stable branches of equilibria are represented in solid and unstable branches by dotted curves. Stable limit cycles are in green and unstable cycles are in blue.  The inset shows the bifurcation structure near the Hopf point $H_2$. $H_1$, $H_2$: Hopf bifurcation, $PD_1, PD_2$: period-doubling bifurcation, $TR$: torus bifrucation, $SN_p$: saddle-node of periodics, $TC$: transcritical bifurcation.}
 \label{one_par_bif}
\end{figure}

\subsection{Numerical investigations of system (\ref{nondim3})}
  Treating the intraspecific competition $h$ as the primary bifurcation parameter and the predation efficiency $\beta_1$ as the secondary parameter, system (\ref{nondim3}) was analyzed in  \cite{Sadhunew}. 
  In this paper, we will treat $h$ as the input parameter and  
  fix the other parameter values to
 \bes
\label{parvalues}   \beta_1=0.1923,  \  \beta_2 =0.6,\ d_1=0.4, \ d_2=0.21,\ \alpha_{12} =0.5, \ \alpha_{21}=0.1,\  \zeta =0.01.
 \ees

Using XPPAUT \cite{Ermen}, we compute a one-parameter bifurcation diagram as shown in figure \ref{one_par_bif}, where the maximum and minimum values of $x$ are considered along the vertical axis.
We note from figure \ref{one_par_bif}  that the boundary equilibrium state $E_{xy}$ is always unstable (exists as a saddle in this parameter range), whereas the stabilities of  the boundary equilibrium $E_{xz}$ and  the coexistence equilibrium state $E^*$  change with $h$. For smaller values of $h$, $E_{xz}$ exists as an unstable focus/node until it undergoes a supercritical Hopf bifurcation at $H_1\approx 0.0613$,  giving birth to a family of stable periodic orbits that lie in the invariant $xz$-plane, and switches to a locally asymptotically stable node/focus. The coexistence equilibrium $E^*$ persists as a saddle or a saddle-focus with one positive and two complex (with negative real parts) eigenvalues for $h< 0.2646$ until it undergoes a subcritical Hopf bifurcation $H_2\approx 0.2646$ and switches to an unstable focus. The two equilibria branches corresponding to $E^*$ and $E_{xz}$ undergo a transcritical bifurcation $TC \approx 0.3577$, past which  $E_{xz}$ exists as an unstable focus/node and $E^*$ as a stable node.  Nearby $H_2$, a subcritical torus bifurcation occurs at  $TR \approx 0.2647$,  resulting into the stabilization of the saddle cycles $\Gamma_h$  born at $H_2$.  On further increasing $h$, the  family of periodic attractors $\Gamma_h$  loses stability at a period-doubling bifurcation $PD_1 \approx 0.272$ and re-gains stability at a second period-doubling bifurcation $PD_2 \approx 0.3986$. Mixed-mode oscillations \cite{BKW, DGKKOW} are observed in this regime. Soon after, the family experiences  a saddle-node bifurcation  $SN_p\approx 0.43143$, after which relaxation oscillations are observed. Further details of the bifurcation diagram are beyond the scope of this paper.

An FSN II bifurcation occurs at $\bar{h}\approx 0.2656$ where the coordinates of the equilibrium $E^*$  are $(\bar{x}, \bar{y}, \bar{z}) \approx (0.2987, 0.1167, 0.4167)$.   It can be checked that system (\ref{nondim3}) satisfies (H3)-(H7) at  $(\bar{x}, \bar{y}, \bar{z}, \bar{h})$.  Note that $H_2$ and $TR$ lie in an $O(\zeta)$ neighborhood of $\bar{h}$. In a small parameter regime past $TR$, where $E^*$ is a saddle-focus and $E_{xz}$ is locally asymptotically stable,  the system exhibits  bistability between  $E_{xz}$ and $\Gamma_h$ as shown in figures \ref{timeseries_all} and \ref{bistable_phasespace} (also see figure \ref{initial_transients}). For the parameter values chosen in figure \ref{timeseries_all}, the eigenvalues of the linearized system at $E^*$ are $\lambda_{1, 2} \approx - 0.02 +3.99 i$, and $\lambda_3 \approx 0.04$, and thus hypothesis (Q1) is met. We note that $\Gamma_h$ has stabilized via a torus bifurcation and thus (Q2) holds. Furthermore, as shown in figure \ref{geom_confg} and figure \ref{bistable_phasespace}, the attractors $E_{xz}$ and $\Gamma$ lie on the opposite sides of the stable eigenplane of $E^*$ and it turns out that the lower branch of the unstable manifold of $E^*$ approaches $E_{xz}$ while the upper branch spirals onto the stable manifold of $\Gamma_h$. Thus system (\ref{nondim3}) also satisfies hypotheses (Q3)-(Q5). 

    \begin{figure}[h!]     
  \centering 
{\includegraphics[width=12.8cm]{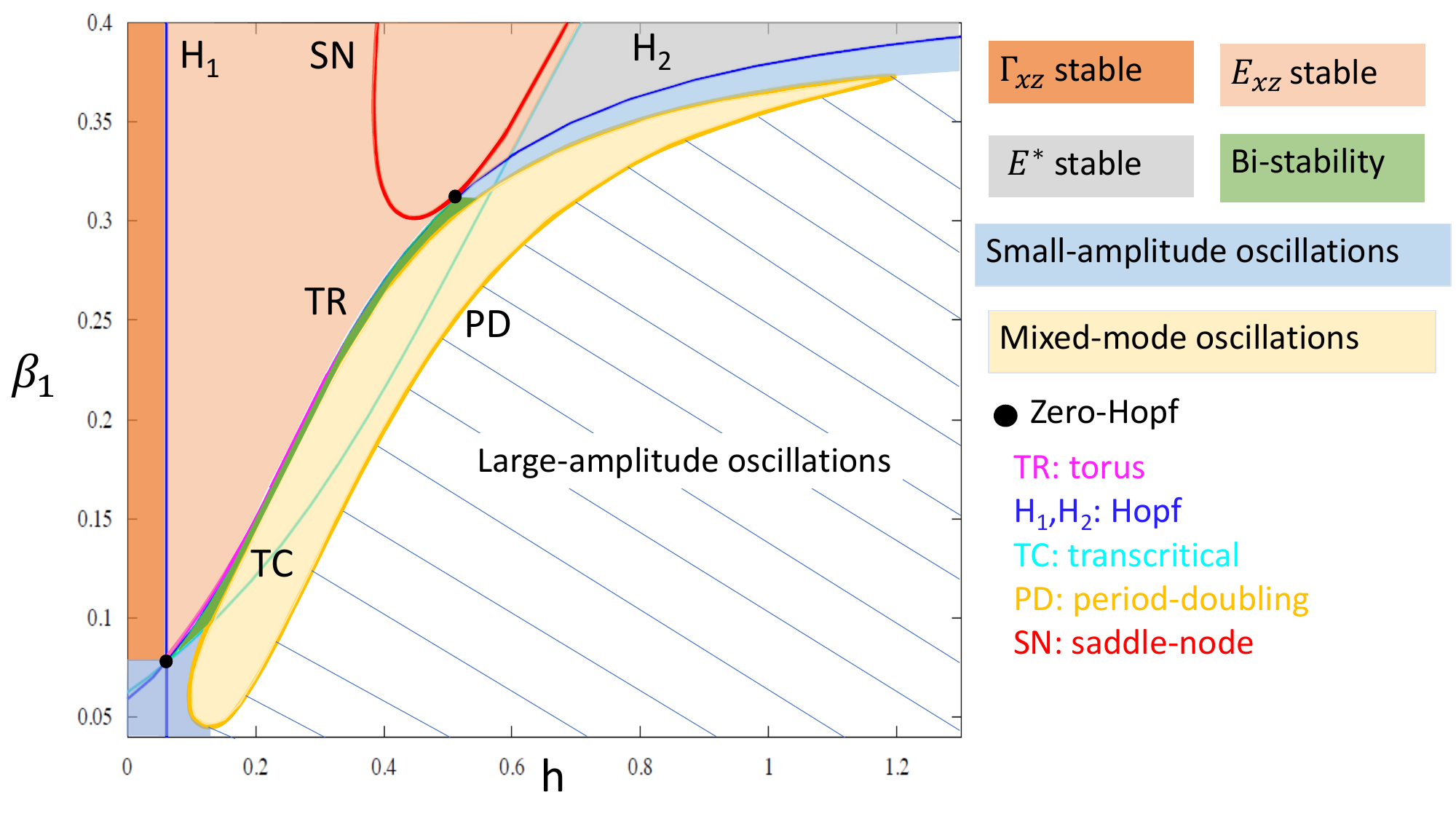}}
 \caption{Two-parameter bifurcation diagram of system (\ref{nondim3}) in $(\beta_1, h)$ space. $E^*$ : coexistence equilibrium, $E_{xz}$: boundary equilibrium, $\Gamma_{xz}$: limit cycle in the $xz$ plane.}  
 \label{two_par_bif}
\end{figure}

To visualize transitions between the different dynamical regimes, we consider a two-parameter bifurcation diagram in $(\beta_1, h)$ space as shown in figure \ref{two_par_bif}, where we present the locally asymptotically stable dynamics existing in these regions. Numerical continuations of $H_1$ and $H_2$ in figure \ref{one_par_bif} divide the parameter space into oscillatory and non-oscillatory dynamical regimes based on the stabilities of $E_{xz}$ and $E^*$. The two curves of equilibria $E_{xz}$ and $E^*$ meet along a transcritical bifurcation $TC$, where they exchange their stabilities. A saddle-node bifurcation of $E^*$ occurs along the SN curve. The equilibria lying on the SN curve or in the region enclosed by the SN curve may not be biologically feasible (see \cite{Sadhunew} for more details). Two codimension-two bifurcations, namely zero-Hopf bifurcations, $ZH_1$ and $ZH_2$, shown by closed circles, occur at the intersection between $H_1$ and $TC$ and at the tangential intersection between $H_2$ and $SN$. The criticalities of the Hopf bifurcation curves change at the zero-Hopf points; $H_1$ is supercritical above $ZH_1$ and subcritical below it, while $H_2$ is supercritical to the right of $ZH_2$  and subcritical to its left.  The two zero-Hopf points are connected by a  torus bifurcation TR curve that emanates from $ZH_2$ and terminates at $ZH_1$. Continuation of the period-doubling bifurcations $PD_1$ and $PD_2$ in figure \ref{one_par_bif}  marks the boundary of the regime where mixed-mode oscillations persist, separating the small-amplitude oscillations born from the TR bifurcation from the large-amplitude oscillations. A bistable behavior is observed in the region enclosed by the TR and PD curves, which is our regime of interest in this paper.

\begin{remark}\label{zerphopf} The parameter regime chosen in system (\ref{nondim3}) lies in a vicinity of the curve of torus bifurcation that emanates from the zero-Hopf/fold-Hopf bifurcation point in the parameter space. The proximity to such a codimension-two bifurcation point leads to long transients that exhibit quasi-stationary dynamics, featuring very slowly decaying amplitude with rapid oscillations, in the vicinity of $E^*$  as seen in figure \ref{timeseries_all}.
\end{remark}

%
%
 In system  (\ref{nondim3}) we noted that in a neighborhood of singular Hopf bifurcation, the invariant manifolds of the saddle-focus equilibrium play central roles in organizing the bistable dynamics.  One way of analyzing the bistable behavior is by numerical computation of the invariant manifolds of the equilibria or the limit cycle and studying the dynamics generated by their interaction. Computations of such manifolds are numerically challenging as they involve stiffness related issues. Another approach is to reduce the system to a topologically equivalent form near the FSN II point on which a simpler geometric treatment  can be applied. In this paper, we will take the latter approach. To characterize the local dynamics  near $E^*$, we will reduce system  (\ref{normal1})  to its normal form, which will be valid in a small neighborhood of $E^*$,  and analyze it. The analysis will be then used to find an early warning sign of an impending population shift. We note that the local analysis performed near the singular Hopf point corresponds to the dynamics of the blown-up vector field on the central chart in the context of geometric desingularization \cite{ks1, ks2}. For orbits that start far from the vicinity of the stable manifold of $E^*$, one needs to consider appropriate phase-directional charts and define suitable global return maps to connect with the dynamics on the central chart. In this paper, we only focus on dynamics that start in a vicinity of the stable manifold of $E^*$.

\section{Normal form near the singular Hopf bifurcation}
\label{sec:normalform}

  A normal form for singular Hopf bifurcation in one-fast and two-slow variables has been explicitly derived by Braaksma \cite{BB}.  We will follow \cite{BB} to reduce system (\ref{normal1}) to its normal form (also see \cite{sadhutrans}).  The reduction allows us to explicitly calculate Hopf bifurcation analytically as stated in the next theorem.

\begin{theorem}
\label{normal} Under assumptions (H3)-(H7),  system (\ref{normal1}) can be written in the normal form:
\begin{eqnarray}\label{normal2}    
\left\{
\begin{array}{ll} \frac{du}{d\tau} &=v+\frac{u^2}{2}+\delta (\alpha(h) u +F_{13} uw +\frac{1}{6}F_{111}u^3)+O(\delta^2)\\
   \frac{dv}{d\tau} &=-u+O(\delta^2) \\
    \frac{dw}{d\tau} & = \delta (H_3 w +\frac{1}{2} H_{11} u^2) +O(\delta^2),
       \end{array} 
\right. 
\end{eqnarray}
where $\tau=\frac{s}{\delta}$ with $\delta = \frac{\sqrt{\zeta}}{{\omega}}$ and $\omega$, $F_{13}$, $F_{111}$, $H_3$, $H_{11}$, and $\alpha(p)$ are given in the Appendix.
The normal form is valid for $(x, y, z, h)=(\bar{x}+O(\sqrt{\zeta}), \bar{y}+O(\zeta) , \bar{z}+O(\zeta), \bar{h}+O(\zeta))$.  Furthermore, system (\ref{normal1}) undergoes a Hopf bifurcation at $h=\bar{h}+\zeta A+O(\zeta^{3/2})$ for sufficiently small $\zeta>0$, where $A$ is the solution of equation (\ref{appnd1}) given in the Appendix.
The Hopf bifurcation is super(sub)critical if the first Lyapunov coefficient
\bes
\label{lyap}
l_1(0)=\frac{\delta}{4}\Big(\frac{1}{2}F_{111}- \frac{F_{13}H_{11} }{H_3}\Big)<(>) 0.
\ees
\end{theorem}

We refer to the work of Braaksma (Theorems 1 and 2 in \cite{BB}) for the detailed proof.   
With extensive algebraic calculations, we can write $(u,v,w)$ in terms of the original coordinates $(x,y,z)$ as given by  (\ref{convert}) in the Appendix. 

 \subsection{Analysis of the normal form} 
We will employ system (\ref{normal2}) to study the dynamics of system (\ref{normal1}) near the equilibrium $E^*=(x_0(p), y_0(p), z_0(p))$  for $p$ in an $O(\zeta)$ neighborhood of $\bar{p}$. In the normal form variables, $E^*$ is mapped to the origin and will be denoted by $q_e$. We note that the eigenvalues of the variational matrix of (\ref{normal2}) at the equilibrium $q_e$ up to higher order terms are 
\bes
\label{eig1}
\rho_1 = \delta H_3,\ \rho_{2,3} = \frac{1}{2} \Big[\alpha \delta \pm i \sqrt{4- \alpha^2 \delta^2}\Big].
\ees
 If $H_3< 0$, then $q_e$ is a stable node or a stable spiral for $\alpha<0$, while it is a saddle-focus with two-dimensional unstable and one-dimensional stable manifold for $0<\alpha<2/\delta$. On the other hand,  if $H_3> 0$, then $q_e$ is an unstable node/spiral for $0<\alpha<2/\delta$, while it is a saddle-focus with two-dimensional stable and one-dimensional unstable manifold for $\alpha<0$.  
 
 Recalling that (\ref{normal1}) satisfies (Q1)-Q5) near $\bar{h}$, it follows that system (\ref{normal2}) also satisfies those hypotheses near $q_e$.  By (Q1),  $q_e$ switches from an unstable focus to a saddle-focus at $\alpha=0$ and the Hopf bifurcation is subcritical, hence we must have $H_3>0$ and  $l_1(0)>0$ in (\ref{lyap}). By (Q3), the flow generated by (\ref{normal2}) on the tangent plane of the stable manifold of $q_e$ points towards the lower-half space, hence we must also have $H_{11}<0$. Let $\Gamma^*_{\alpha}$ be the image of $\Gamma(p)$ in the normal form variables, where we consider $\alpha \in (\alpha_s, \alpha_e)$ for some  $\alpha_s, \alpha_e<0$ with $O(H_3)=\alpha_s, \alpha_e$ so that $\Re(\rho_{2,3})=O(\rho_1)$. Without loss of generality, we will assume that $\Gamma^*_{\alpha}$ lies on the upper-half space  $\{w>0\}$. We note from hypothesis (Q4) that for all $\alpha \in (\alpha_e, 0)$,  the unstable manifold of $q_e$ tends to the stable manifold of  $\Gamma^*_{\alpha}$ in the upper-half space $\{w>0\}$.  
 Finally, if we denote the image of $\tilde{\Gamma}(p)$ by $\tilde{\Gamma}^*_{\alpha}$, then it follows from (Q5) that $\tilde{\Gamma}^*_{\alpha}$ is not necessarily an attractor of system (\ref{normal2}) as it lies outside the domain of validity of the normal form variables. Hence we will define $\tilde{\Gamma}^*_{\alpha}$ to be the singleton set $\{(0, 0, -\infty)\}$ in the $(u, v, w)$ coordinates. An  illustration of a possible bistable behavior that system (\ref{normal2}) can exhibit is shown in figure \ref{normal_ratio_bdd_unbdd}.
 
  \begin{figure}[h!]     
  \centering 
     \subfloat[A trajectory approaching the limit cycle (blue).]
      {\includegraphics[width=6.0cm]{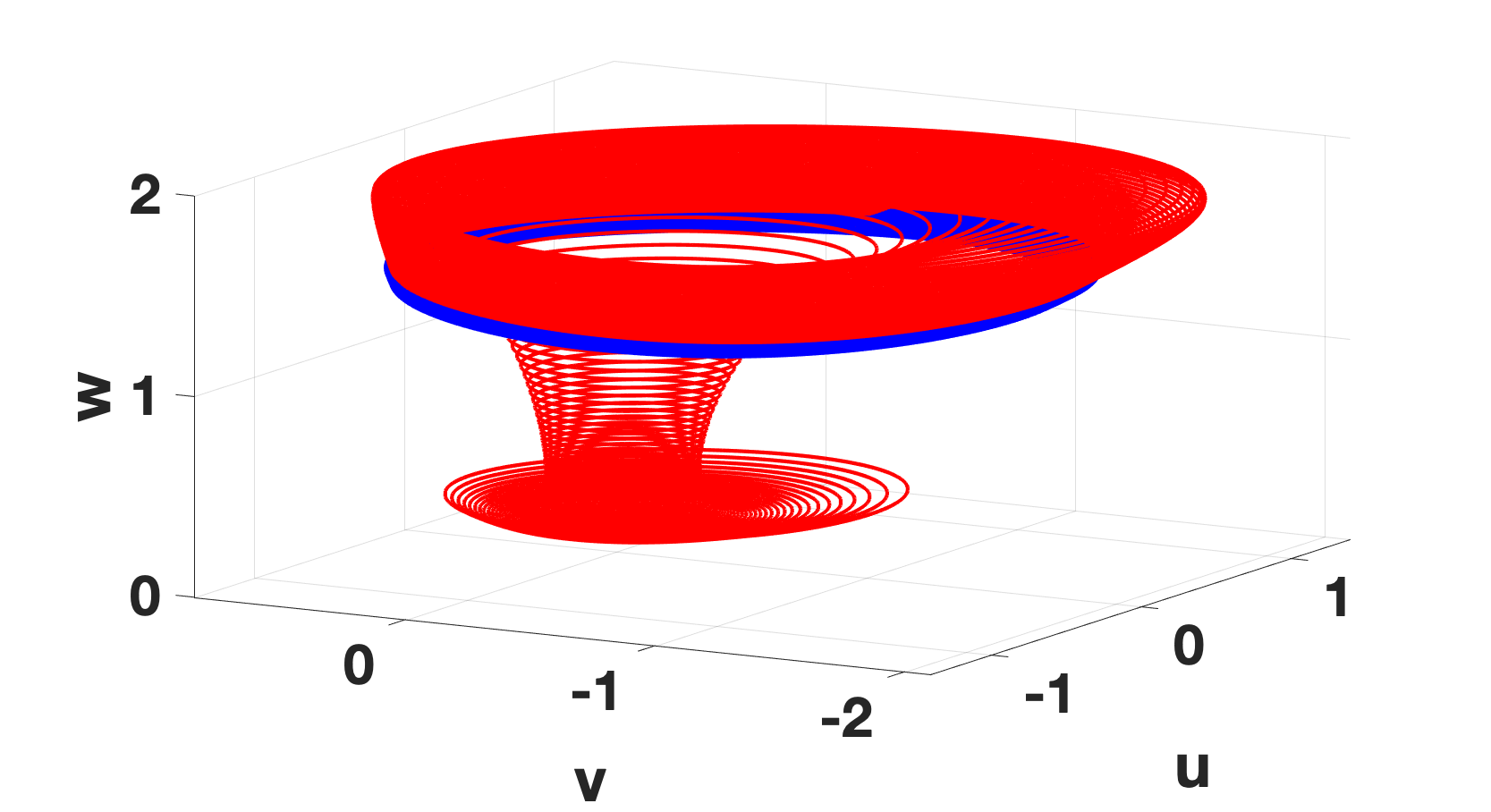}}\quad
      \subfloat[A trajectory approaching $(0, 0, -\infty)$]{\includegraphics[width=6.0cm]{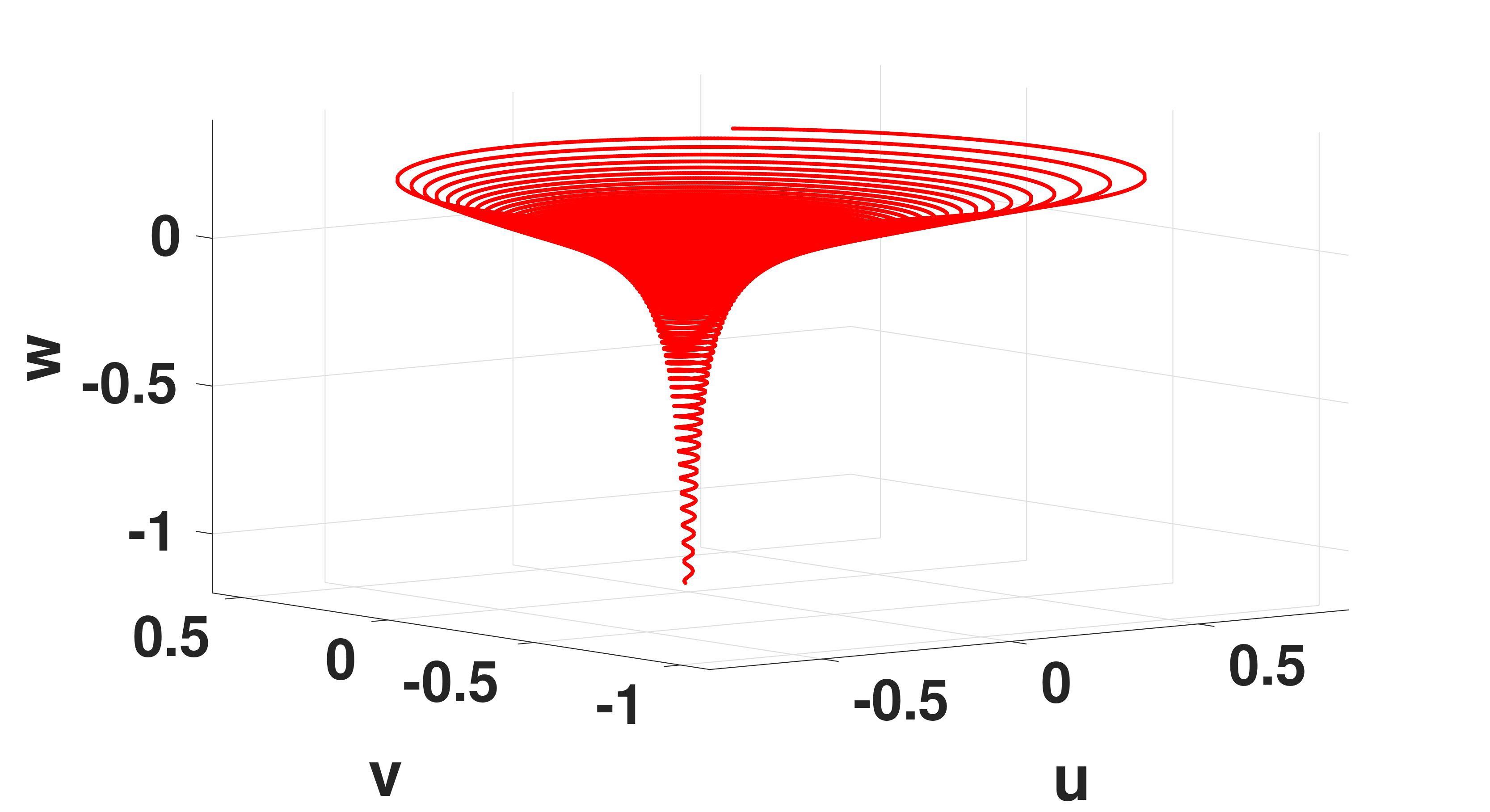}}
  \caption{An illustration of bistable behavior exhibited by system (\ref{normal2}) for suitable parameter values. Both solutions start above the plane $\{w=0\}$; one of them approaches $\Gamma$ (shown in blue) and other spirals down the $w$-axis  and approaches $\tilde{\Gamma}$ in the normal form coordinates. (A)  Initial conditions: $(0.452, 0.432, 0.329)$.  (B)  Initial conditions: $(0.452, 0.432, 0.249)$.}
  \label{normal_ratio_bdd_unbdd}
  \end{figure}

  \subsubsection{Linear Analysis near $q_e$}
  For $\alpha<0$ and $H_3>0$,  the flow generated by (\ref{normal2}) up to the first order is given by 
 \begin{eqnarray}\label{flow1}    
\left\{
\begin{array}{ll} u(\tau) &= Ae^{\frac{\alpha \delta \tau}{2}}  \sin(\vartheta \tau +\phi_1),\\
\label{flow} v(\tau) &= A e^{\frac{\alpha \delta \tau}{2}} \sin(\vartheta \tau +\phi_2),\\
w(\tau) &= e^{\delta H_3 \tau} \Big(w_0+\frac{\delta H_{11}}{2} \Big(C+\frac{4\vartheta^2 u_0^2+(2v_0 +\alpha \delta u_0 )^2}{8 \vartheta^2 \delta (\alpha-H_3)}\Big)\Big)\\
 &+ \frac{\delta}{2} H_{11}e^{\alpha \delta \tau} \Big[C \cos(2\vartheta \tau)+D \sin(2\vartheta \tau) 
+\frac{4\vartheta^2 u_0^2+(2v_0 +\alpha \delta u_0 )^2}{8 \vartheta^2 \delta (\alpha-H_3)} \Big],
   \end{array} 
\right. 
\end{eqnarray}
where $(u(0), v(0), w(0))=(u_0, v_0, w_0)$ and \bess \vartheta &=& \sqrt{1-\frac{\alpha^2 \delta^2}{4}}, \ A= \frac{1}{\vartheta}\sqrt{u_0^2 +\alpha\delta u_0 v_0+v_0^2}, \ \phi_1=\tan^{-1}\Big(\frac{2\vartheta u_0}{2v_0+\alpha\delta u_0} \Big), \\
 \phi_2 &= &\tan^{-1}\Big(-\frac{2\vartheta v_0}{2u_0+\alpha\delta v_0} \Big), \\
C &=& \frac{\delta(\alpha-H_3) (4\vartheta^2 u_0^2-(2v_0 +\alpha \delta u_0 )^2)- 8 \vartheta^3 u_0(2v_0+\alpha \delta u_0) } {8\vartheta^2 (4\vartheta^2+\delta^2(\alpha-H_3)^2)}, \\
D &=& \frac{ 4\vartheta^2 u_0^2-(2v_0 +\alpha \delta u_0 )^2 +2\delta(\alpha-H_3)  u_0(2v_0+\alpha \delta u_0) } {4 \vartheta (4\vartheta^2+\delta^2(\alpha-H_3)^2)}.
\eess

 \subsubsection{Parametrization of the slow variable in system (\ref{normal2})}
For a fixed $\alpha<0$ such that $|\alpha|=O(H_3)$, one may replace the slow variable $w$ in system (\ref{normal2}) by a  parameter $\lambda$ to obtain insight into the dynamics of the fast variables. 
The fast variables $(u, v)$ are then governed by the parametrized  system 
\begin{eqnarray}\label{normal_par}    
\left\{
\begin{array}{ll} \frac{du}{d\tau} &= \delta (\alpha + F_{13} \lambda) u+v+\frac{u^2}{2}+ \frac{\delta}{6}F_{111}u^3 \\
   \frac{dv}{d\tau} &=-u
       \end{array} 
\right. 
\end{eqnarray}
up to $O(\delta^2)$. Linearization of (\ref{normal_par}) yields that the eigenvalues at the origin are 
\bes  \label{eigenval} \sigma_{1,2}(\lambda) = \frac{1}{2} \Big(\delta (\alpha +F_{13} \lambda) \pm \sqrt{\delta^2 (\alpha +F_{13} \lambda)^2-4} \Big) +O(\delta^2).
\ees
For $\lambda>0$, we note from (\ref{eigenval}) that the origin $(0, 0)$ always remains asymptotically stable if $F_{13}\leq 0$, which then would imply that system (\ref{normal2}) cannot approach $\Gamma_{\alpha}^*$, the only attractor of the system that lies in the upper half space. Hence we must have $F_{13}>0$.  
  A Hopf bifurcation of (\ref{normal_par})   occurs at $\lambda=\lambda_H(\alpha)= -\alpha/F_{13}$ with the first Lyapunov coefficient $\delta F_{111}/8$ \cite{sadhutrans}. 
   The origin $(0,0)$ is asymptotically stable for $0<\lambda<\lambda_H(\alpha)$ and unstable for $\lambda>\lambda_H(\alpha)$. To ensure that the oscillations of (\ref{normal_par})  are bounded, we also must have that $F_{111}<0$ with $|F_{111}|/6 =O(F_{13})$.

Henceforth throughout the paper, we will assume that $F_{111}<0$, $F_{13}>0$, $H_3>0$ and $H_{11}<0$ such that $l_1(0)>0$ in (\ref{lyap}).
  
   \subsection{Analysis of the bistable behavior in system (\ref{normal2})} 
 Fix $\alpha<0$ with $|\alpha|=O(H_3)$ and consider the dynamics of system (\ref{normal2}) near $q_e$.  For a trajectory with $u^2(0), v^2(0)=O(\delta)$ and $0<w(0)=O(\delta)$ that lies in a very close neighborhood of the local stable manifold of $q_e$, $W^s_{\mathrm{loc}}(q_e)$, it may either spiral away from $q_e$ and approach the limit cycle  $\Gamma_{\alpha}^*$, or it may cross the plane $\{w=0\}$ and eventually approach $(0,0, -\infty)$ as $\tau\to \infty$. Note that  if  $w(\tau)\geq -\frac{H_{11}}{2H_3} (u^2+v^2)(\tau)$ as the envelope of $(u^2+v^2)(\tau)$  decreases,  then we have from  (\ref{normal2}) that
$\frac{dw}{d\tau} >  -\frac{1}{2}\delta H_{11} v^2 \geq 0$, 
implying that the trajectory must spiral up and move away from $q_e$. On the other hand, if $w(\tau) < -\frac{H_{11}}{2H_3} (u^2+v^2)$  as long as the trajectory spirals inwards, then the fact that $u^2(\tau)+v^2(\tau)=O(e^{\alpha \delta \tau})$  (follows from (\ref{flow1}))  will imply that  $w(\tau)< \kappa e^{\alpha \delta\tau}$ for some $\kappa>0$.  Since $q_e$ is unstable, $w(\tau)$ cannot approach zero as $\tau \to \infty$. Hence $w(\tau)$ must cross the plane $\{w=0\}$ at some finite time $\tau=\tilde{\tau}$. The invariance of the lower-half space $\{w\leq 0\}$ would then imply that the trajectory approaches $(0, 0, -\infty)$ as $\tau\to \infty$ in this case.

Next, we will analyze the boundary of the basins of attraction of $\Gamma_{\alpha}^{\star}$ and $\tilde{\Gamma}_{\alpha}^{\star}$. We consider the set $\Omega_{\alpha}= \{(u,v,w): w < -\frac{H_{11}}{2H_3} (u^2+v^2)\}$ and show that its complement restricted to a small neighborhood of $q_e$ separates solutions approaching $\Gamma_{\alpha}^*$ from $\tilde{\Gamma}^{\star}$. We first note from (\ref{flow1}) that  solutions in a vicinity of $q_e$ feature rapid oscillations with slow variation in their amplitudes along $W^s_{\mathrm{loc}}(q_e)$, while slowly evolving along$W^u_{\mathrm{loc}}(q_e)$ for $\alpha<0$, $|\alpha|=O(H_3)$. Hence, to understand the intrinsic behavior of the trajectories near $q_e$  as they make a slow passage through it,  we will exploit the timescale separation in the dynamics by  applying the method of averaging \cite{GH, V} and  study the slow evolution of the system near $\Omega_{\alpha}$ restricted to a  neighborhood of $q_e$. To this end, we consider a solution that starts in $\Omega_{\alpha}$ near $W_{\mathrm{loc}}^s(q_e)$. To determine whether such a solution will enter into the funnel $\mathbb{R}^3 \setminus \Omega_{\alpha}$ and eventually approach $\Gamma_{\alpha}^*$, or remain in $\Omega_{\alpha}$ for all $\tau>0$ and approach $(0, 0, -\infty)$, we will consider the  moving averages of $u$, $v$ and $w$ over intervals of length $l$ (to be defined later), denoted by $\bar{u}$, $\bar{v}$ and  $\bar{w}$ respectively, where the moving average of a function $g$ is defined by
\bes \label{movingavg}
\bar{g} (\tau) = \frac{1}{l} \int_{\tau}^{\tau+l}g(s) \ ds \ \textrm{for} \ \tau\geq 0.
\ees

 \subsubsection{Derivation and analysis of the averaged system}  Fix $\alpha<0$ with $|\alpha|=O(H_3)$. Under the assumption
 \[
 \textnormal{(P1)} \   F_{13}>0,\  F_{111}<0, \ H_3>0,\  H_{11}<0 \textnormal{\ such\  that\ } l_1(0)>0 \textnormal{\ in } (\ref{lyap}),\]
 we consider a solution $(u(\tau), v(\tau), w(\tau))$ of system  (\ref{normal2}) with initial values $(u_0, v_0, w_0) \in \Omega_{\alpha}$ such that  $u_0, v_0=O(\sqrt{\delta})$, $w_0=O(\delta)$, where $0<w_0 <-\alpha/F_{13}$ such that the following holds:\\
 
(P2) {\emph{$(u(\tau), v(\tau), w(\tau)) \in \Omega_{\alpha}$ for all $\tau\in [0, \tau_N]$ and $w(\tau)$ has a decreasing envelope on $[0, \tau_N]$ with $w({\tau}_{N})>0$, where $\{\tau_i\}_{i=1}^N$ 
is an increasing sequence of locations of relative maxima of $u(\tau)$ with $\{u(\tau_i)\}_{i=1}^N$ decreasing}}.\\

 We will now derive and analyze the equation of $\bar{w}$ to find a set of sufficient conditions that distinguish the long-term behavior of the solutions. Integrating the third equation of (\ref{normal2}) over an interval $[\tau, \tau+l]$ and using the fact that $\bar{w'}=(\bar{w})'$, we note that  $\bar{w}(\tau)$ satisfies
\bes
\label{wbar_eqn}
\frac{d\bar{w}}{d\tau} = \delta H_3 \bar{w} + \frac{\delta}{2} H_{11}\bar{u^2}.
\ees
 Solving (\ref{wbar_eqn}) yields that
\bes
\label{env_w}
\bar{w}(\tau) = \bar{w}(\tau_0) e^{\delta H_3 (\tau-\tau_0)}+\frac{\delta H_{11}}{2} \int_{\tau_0}^{\tau}e^{\delta H_3 (\tau-s)}\bar{u^2}(s) \ ds
\ees
for $\tau>\tau_0$, where $\tau_0\geq 0$  is an arbitrary chosen time.  Let $I=[\tau_1, \tau_N]$ and $p$ be the average period of oscillations of $u(\tau)$ in the interval $I$. Note that as long as the trajectory spirals inwards towards  $q_e$, the dynamics of $u$ and $v$ can be approximated by  (\ref{normal_par}), and thus  $\bar{u^2}$, $\bar{v^2}$ will have the form $(a+c\sin(2\vartheta \tau) +d\cos(2\vartheta \tau))e^{b(\tau)\tau}+o(1)$ as $\delta \to 0$ for some $a>0$,  $c, d \approx 0$ and $b(\tau)<0$, a slowly varying function such that $b(\tau) \to \alpha \delta$ as $w(\tau) \to 0$. Here, $\vartheta  \approx 2\pi/p$ and $a, b(\tau), c$, and $d$ depend on the initial conditions.  For simplicity, we will assume that $b(\tau)= b_2$, where $b_2<0$ is such that  $|b_2- \alpha \delta|=O(\delta^m)$, $m\geq 2$. 
Decomposing $\bar{u^2}(\tau)$ into the sum of its average and fluctuations from the mean, with the aid of (\ref{flow1}) we can write $\bar{u^2}(\tau) = \bar{u^2}_{\textrm{est}}(\tau) +O(\delta^2)$, where
\bes \label{uest} \bar{u^2}_{\textrm{est}}(\tau) =\bar{u^2}_{\textrm{base}}(\tau)+\bar{u^2}_{\textrm{osc}}(\tau)\ \textnormal{with}\ees 
\bes \label{ubase}
\bar{u^2}_{\textrm{base}}(\tau) & = & b_1 e^{b_2(\tau-\tau_1)},\\
\nonumber  \bar{u^2}_{\textrm{osc}}(\tau) &=&  e^{b_2(\tau - \tau_1)}
 \Big(\gamma_1\sin(2\vartheta (\tau -\tau_1)) +\gamma_2 \cos (2\vartheta (\tau- \tau_1))\Big), \\
 \nonumber  b_1 & = & \frac{A^2(1-e^{-pb_2})}{2pb_2}, \ b_2= \alpha \delta+O(\delta^m), \ \textrm{and} \  \gamma_1, \gamma_2=O(\delta^m), \ m\geq 2.\ees
Recalling that $u_0, v_0=O(\sqrt{\delta})$, hence we must have that $b_1 =O(\delta)$. Next we choose $\tau_0=\tau_1$ and $l=p$ in (\ref{env_w}). Ignoring the contribution from $ \bar{u^2}_{\textrm{osc}}(\tau)$ in $\bar{u^2}(\tau)$ and using the form of $\bar{u^2}_{\textrm{base}}(\tau)$ in (\ref{ubase}), we then have from (\ref{env_w}) that $\bar{w} = \bar{w}_{\textrm{base}} +O(\delta^2)$, where
 \bes
\label{west1} \bar{w}_{\textrm{base}}(\tau)= \Big(\bar{w}(\tau_1) - \frac{\delta H_{11}b_1}{2(b_2-\delta H_3)} \Big)e^{\delta H_3 (\tau-\tau_1)}+\frac{\delta H_{11}b_1}{2(b_2-\delta H_3)}e^{b_2(\tau-\tau_1)}.
\ees  
We note that as long as $w(\tau)=O(\delta)$ and $u^2(\tau), v^2(\tau) =O(\delta)$, we have from  (\ref{normal2}) that $w'(\tau)=O(\delta^2)$. Hence, it follows from the Mean Value Theorem that $w(\tau)- \bar{w}(\tau) = \frac{1}{p}\int_{\tau}^{\tau+p} (w(\tau)-w(s))\ ds \leq \frac{p}{2}{\sup}_I|w'| =O(\delta^2)$ and thus $w(\tau) -  \bar{w}_{\textrm{base}}(\tau) =O(\delta^2)$.

\begin{remark}\label{rmk1} The approximation of the decay function $b(\tau)$ by a constant $b_2$  is also valid for $\tau>\tau_N$ as long as the solution spirals inwards and stays in the vicinity of the $\{w=0\}$ plane. Therefore, the expressions in (\ref{uest}) and (\ref{ubase}) are also valid beyond $\tau_N$.
\end{remark}

We will now state our main result.


\begin{theorem} 
\label{bistability} Assume (P1)-(P2). Then for $\delta>0$ sufficiently small, system (\ref{normal2}) approaches $\Gamma_{\alpha}^*$ as $\tau \to \infty$ if 
\bes
\label{condnew}  \frac{\delta H_{11}b_1}{2(b_2-\delta H_3)} + \delta^2< \bar{w}(\tau_1) < -\frac{H_{11}b_1}{2H_3},
\ees
where
    \bes  \label{rel0}
 b_1 &=& \frac{A^2\vartheta}{4\pi b_2 }(1-e^{\frac{-2\pi b_2}{\vartheta}}),\  b_2 = \alpha \delta+O(\delta^m), \ m\geq 2 \\ \nonumber   \textnormal{with} \\
 \label{relA}  A &=& \frac{1}{\vartheta}\sqrt{u_0^2 +u_0 v_0 b_2+v_0^2}, \  \vartheta = \sqrt{1- \frac{\alpha^2 \delta^2}{4}}+O(\delta^2).
\ees
On the other hand, if  $\bar{w}(\tau_1)< \frac{\delta H_{11} b_1 e^{b_2 \tau_1}}{2(b_2-\delta H_3)}$, then
 $(u(\tau), v(\tau), w(\tau)) \to (0,0, -\infty)$ as $\tau \to \infty$. 
\end{theorem}

To prove the theorem, we will need a few lemmas, the proofs of which are presented in the Appendix.


\begin{lemma}
\label{funnel} Assume (P1)-(P2). Suppose that $(u(\tau), v(\tau), w(\tau))$ is a solution of system  (\ref{normal2}) that meets condition (\ref{condnew}). 
Then  $w(\tau)$  attains its global minimum at some ${\tau}_{\mathrm{min}} \geq \tau_N$. 
\end{lemma}

\begin{lemma}
\label{funnel1} Assume (P1)-(P2). Suppose that $(u(\tau), v(\tau), w(\tau))$ is a solution of system  (\ref{normal2}) that meets condition (\ref{condnew}). 
Then  there exist $\tau_a , \tau_b > {\tau}_{\mathrm{min}}$ such that  
$(u(\tau), v(\tau), w(\tau)) \in \mathbb{R}^3 \setminus \Omega_{\alpha}$  for all $\tau \in [\tau_a, \tau_b]$.
\end{lemma}

\begin{lemma} \label{invariance} 
Assume (P1)-(P2). Then the  set $\Omega_{\alpha}$ is positively invariant with respect to the solution  $(u(\tau),v(\tau),w(\tau))$  of system  (\ref{normal2}) 
if $\bar{w}(\tau_1)< \frac{\delta H_{11}b_1}{2(b_2-\delta H_3)}$, where  $b_1$ and $b_2$ are defined as in (\ref{rel0}).  Furthermore, $ (u(\tau),v(\tau),w(\tau)) \to (0, 0, -\infty)$ as $\tau \to \infty$.
\end{lemma}

%
%
{\emph{Proof of Theorem  \ref{bistability}:}} Lemmas \ref{funnel} and  \ref{funnel1}  imply that a solution of (\ref{normal1}) initiated in $\Omega_{\alpha}$ enters into the funnel, $\mathbb{R}^3\setminus \Omega_{\alpha}$, and remains inside it for all $\tau \in [\tau_a, \tau_b]$ if (\ref{condnew}) holds.  Furthermore, it follows from (\ref{normal_par}) that as long as the trajectory spirals inward, it remains in $\mathbb{R}^3\setminus \Omega_{\alpha}$; in particular this occurs for all $\tau>\tau_a$ such that  $w(\tau) <-\alpha/F_{13}$. We also note from Remark \ref{hypth} that  the basin of attraction of  $\Gamma_{\alpha}^*$  
must contain the set $\mathcal{B}:=\{(u, v, w) \in [-\delta, \delta]\times [-\delta, \delta]\times(0, -\alpha/F_{13}]:  w\geq \Theta(u, v, \delta)\} \cap \{ w> -H_{11}/(2H_3)u^2 \}$ for sufficiently small $\delta>0$, where $w=\Theta(u, v, \delta) +O(\delta^2)$ is the local approximation of $W^s_{\textnormal{loc}}(q_e)$.
Since $(\mathbb{R}^3\setminus \Omega_{\alpha})\cap \{w<-\alpha/F_{13}\} \subset \mathcal{B}$ for all $\delta>0$ sufficiently small,  hence  the flow restricted to $\{(u, v, w)\in \mathbb{R}^3\setminus \Omega_{\alpha}: 0<w\leq -\alpha/F_{13}\}$  must approach the attractor  $\Gamma_{\alpha}^*$ as $\tau \to \infty$, which proves the first part of Theorem \ref{bistability}.  The second part of Theorem \ref{bistability}  follows directly from Lemma \ref{invariance}.

 {\hfill \ensuremath{\Box}}

\section{Early warning signals}
\label{sec:early_warning}
In the previous section, we obtained a set of sufficient conditions on $\bar{w}(\tau_1)$  to predict the asymptotic behavior of a trajectory as it approaches  the equilibrium $q_e$. 
In this section, we focus on finding a method of predicting long-term behaviors of trajectories as early as possible during their journeys towards  $q_e$ in the same parameter regime considered in Theorem \ref{bistability}.  The significance of the method lies in finding the shortest time interval over which we can predict a forthcoming sign change in $w(\tau)$ when a trajectory fails to enter into the funnel $\mathbb{R}^3\setminus \Omega_{\alpha}$ as it approaches $q_e$.

Due to the difference in timescales between the frequency and the amplitude of  oscillations of $u(\tau)$, we can choose the initial conditions  in such a way that there exist at least $N\gg 1$ oscillations before the amplitude of $u(\tau)$ decays by a factor of $e$.   Indeed, we note from (\ref{flow1}) that as a solution of (\ref{normal1}) approaches $q_e$, the amplitude of  $u(\tau)$ decays exponentially with the decay function $b(\tau)/2 \approx \alpha \delta /2$, while the frequency of its oscillations $\vartheta  \approx 1 -\frac{1}{8}\alpha^2\delta^2 \approx 1$ for sufficiently small $\delta>0$,  hence one can choose $(u_0, v_0, w_0)$ such that $\tau_N> \frac{\vartheta}{2\pi |b(\tau_N)|}$, where $\tau_i$, $i=1, 2, \ldots, N, \ldots $ are the locations of relative maxima of $u(\tau)$. 
We now define a sequence of intervals $I_i=[\tau_1, \tau_{k+i}]$, $i=1, 2 \ldots N-k$ such that each interval contains at least  $k$ oscillations of $u$, where  $k \ll N$. The goal is to find the smallest interval $I_i$ (and therefore obtain the shortest time) on which a forthcoming population collapse can be accurately predicted.

To this end, let $l_i$ be the average period of oscillations of $u(\tau)$ in the interval $I_i$, i.e.  $l_i = \frac{1}{k+i}\sum_{j=1}^{k+i-1} (\tau_{j+1}-\tau_j)$.   
On each interval $I_i$, 
we will approximate $\bar{u^2}$ by $\bar{u^2_i}_{\textrm{est}}$, where $\bar{u^2_i}_{\textrm{est}}(\tau)=\bar{u^2_i}_{\textrm{base}}(\tau)+\bar{u^2_i}_{\textrm{osc}}(\tau)$ with
\bess
\bar{u^2_i}_{\textrm{base}}(\tau)  = k^i_1 e^{k^i_2 (\tau- \tau_1)}, \ \bar{u^2_i}_{\textrm{osc}}(\tau)=  e^{k^i_2 (\tau- \tau_1)}
 \Big(\gamma^i_1\sin \Big(\frac{\pi (\tau -\tau_1)}{l_i} \Big)+\gamma^i_2 \cos \Big(\frac{\pi (\tau-\tau_1)}{l_i}\Big)\Big),
\eess
\bes \label{seq1}
 k^i_1 = \frac{1}{\vartheta}\sqrt{u^2(\tau_1) + u(\tau_1)v(\tau_1) k^i_2+ v^2(\tau_1)}, \   
 k^i_2= O(\alpha \delta)\ \textnormal{and} \ \gamma^i_1, \gamma^i_2=O(\delta^2). 
 \ees
 Here $0<k_1^i =O(\delta)$ and $k_2^i<0$.  We will assume that $\{k_1^i\}_{i=1}^{N-k}$ is monotonically decreasing  whereas $\{k_2^i\}_{i=1}^{N-k}$ is monotonically  increasing. Furthermore, we will assume that $\{l_i\}_{i=1}^{N-k}$ is nearly constant with $|l_{i+1}-l_i|=O(\delta^m)$, $m\geq 2$. As in Section 3, we will ignore the contribution from $ \bar{u^2_i}_{\textrm{osc}}(\tau)$ and only consider the effect of 
$\bar{u^2_i}_{\textrm{base}}(\tau)$, which represents the base value about which $\bar{u^2}(\tau)$ oscillates in $I_i$. 
Letting  $l=l_i$ in (\ref{movingavg}), we consider the moving average of $w$ restricted to $I_i$.  Using the above form of $\bar{u^2_i}_{\textrm{base}}(\tau)$, we then have from (\ref{env_w}) that  
\bes
\label{wavg}
\bar{w}^i_{\textrm{base}}(\tau)= \Big(\bar{w}(\tau_1) - \frac{\delta H_{11}k^i_1}{2(k^i_2-\delta H_3)}\Big) e^{\delta H_3 (\tau-\tau_1)}+\frac{\delta H_{11}k^i_1}{2(k^i_2-\delta H_3)}e^{k^i_2(\tau-\tau_1)}, 
\ees
where $\bar{w}^i_{\textrm{base}}(\tau)$ is an approximation of $\bar{w}(\tau)$ with $|\bar{w}(\tau)-\bar{w}^i_{\textrm{base}}(\tau)|=o(1)$ on the interval $I_i$.
For each $i$ between $1$ and $N-k$, we define a critical curve $\bar{w}^i_{\textrm{crit}}(\tau)$  by  \bes
\label{crit_curve} \bar{w}^i_{\textrm{crit}}(\tau) = \frac{\delta H_{11}k^i_1 e^{k^i_2(\tau-\tau_1)}}{2(k^i_2-\delta H_3)},\ \tau \in I_i.
\ees
 We will examine the behavior of $\bar{w}^i_{\textrm{base}}(\tau)$  relative to its position with respect to  $\bar{w}^i_{\textrm{crit}}(\tau)$ as stated in the next two propositions. 

\begin{proposition} 
\label{envelope_min} 
Assume (P1).  Let $\{I_i\}_{i=1}^{N-k}$ be a  sequence of nested intervals  defined by $I_i=[\tau_1, \tau_{k+i}]$,  where $\{ \tau_i\}_{i=1}^{N}$ is an increasing sequence of locations of relative maxima of $u(\tau)$ such that $\{u(\tau)\}_{i=1}^N$ is decreasing and
 $k \ll N$ is a fixed integer greater than $4$ with $N=O(\frac{k}{\delta})$.  Suppose that $\{k_1^i\}_{i=1}^{N-k} $ and $\{k_2^i\}_{i=1}^{N-k}$, defined by (\ref{seq1}), are respectively decreasing and increasing sequences of real numbers  such that $\bar{u^2}(\tau)=k_1^ie^{k_2^i (\tau - \tau_1)}+o(1)$ on $I_i$. Then 
 $w(\tau)$ attains its global minimum in $I_{N-k}$ if  \bes
\nonumber  \bar{w}(\tau) & > & \bar{w}^1_{\textrm{crit}}(\tau) \ \textrm{on} \  I_1 \ \textrm{and}  \\
\label{cond}   \bar{w}(\tau) & > & \bar{w}^i_{\textrm{crit}}(\tau)\ \textrm{on} \  I_i \setminus I_{i-1}\ \textrm{for all} \ 2\leq i \leq N-k,
\ees
where 
$\bar{w}^i_{\textrm{crit}}(\tau)$ is defined by (\ref{crit_curve}).

\end{proposition}

\begin{proof} Since $\bar{w}(\tau)= \bar{w}^i_{\textrm{base}}(\tau)+o(1)$, it follows that $\bar{w}^i_{\textrm{base}}(\tau)$   also satisfies (\ref{cond}) if $\bar{w}(\tau)$ satisfies condition (\ref{cond}).  It  can be easily verified from (\ref{wavg}) that if $\bar{w}^i_{\textrm{base}}(\tau) >  \bar{w}^i_{\textrm{crit}}(\tau)$ then $\bar{w}^i_{\textrm{base}}(\tau)$ attains a minimum at $\tau = \bar{\tau}_i$, where
\bess
\bar{\tau}_i = \frac{1}{\delta H_3-k^i_2}\ln \Big(\frac{H_{11}k^i_1k^i_2e^{(\delta H_3 -k^i_2) \tau_1}}{H_3(2\bar{w}(\tau_1)(\delta H_3-k^i_2)+\delta H_{11}k^i_1)} \Big).
\eess
 Let $M$ be the smallest integer between $1$ and $N-k$ such that $\bar{\tau}_M \in I_M$. Then 
 $\bar{w}$ will also attain its  minimum near $\bar{\tau}_M$. 
 Consequently, $w$ will attain its global minimum near $\bar{\tau}_M \in I_{N-k}$.
\end{proof}


In the next proposition, we will show that the trajectory continues to spiral down and crosses the plane $\{w=0\}$ if  $\bar{w}^i_{\textrm{base}}(\tau)$ is below $\bar{w}^i_{\textrm{crit}}(\tau)$ on $I_i$ for some $i$.

\begin{proposition}
\label{envelope_extinct} Let $\{I_i\}_{i=1}^{N-k}$ be a sequence of nested intervals and $\{k^i_1\}_{i=1}^{N-k}$, $\{k^i_2\}_{i=1}^{N-k}$ be sequences of real numbers as defined in Proposition \ref{envelope_min}. Then there exists  $\tilde{\tau}_e \in I_{N-k}$ such that $w(\tilde{\tau}_e)=0$ with $w'(\tilde{\tau}_e)<0$ if $\bar{w}(\tau)<\bar{w}^{i_0}_{\textrm{crit}}(\tau)$ on $({\tilde{\tau}}_{i_0}, \tau_{k+i_0}]$ for some $1\leq i_0 \leq N-k$ and ${\tilde{\tau}}_{i_0} \in I_{i_0}$, where $\bar{w}^i_{\textrm{crit}}(\tau)$ is defined by (\ref{crit_curve}).
\end{proposition}

\begin{proof}  We first note that the monotonic properties of $\{k_1^i\}$ and $\{k_2^i\}$ imply that $\bar{w}^i_{\textrm{crit}}(\tau)< \bar{w}^{i+1}_{\textrm{crit}}(\tau)$ on $I_i$. Moreover, if $\bar{w}^i_{\textrm{base}}(\tau)< \bar{w}^i_{\textrm{crit}}(\tau)$ for some $i$, then it follows from (\ref{wavg}) that $\bar{w}^i_{\textrm{base}}(\tau)$ decreases and eventually becomes negative at $\tau =\tilde{\tau}_i$, where
 \bes
 \label{extinct}
\tilde{\tau}_i = \frac{1}{\delta H_3-k^i_2}\ln \Big(\frac{\delta H_{11}k^i_1e^{(\delta H_3 -k^i_2) \tau_i}}{\delta H_{11}k^i_1 +2\bar{w}(\tau_1)(\delta H_3-k^i_2)} \Big).
\ees
If there exists some point ${\tilde{\tau}}_{i_0} \in I_{i_0}$ such that $\bar{w}(\tau)<\bar{w}^{i_0}_{\textrm{crit}}(\tau)$ on $({\tilde{\tau}}_{i_0}, \tau_{k+i_0}]$ for some $1\leq i_0\leq N-k$, then the monotonicity  of the family $\{\bar{w}^{i}_{\textrm{crit}}(\tau)\}_{i=1}^{N-k}$ yields that $\bar{w}(\tau)<\bar{w}^{i}_{\textrm{crit}}(\tau)$ on $({\tilde{\tau}}_{i_0}, \tau_{k+i}]$ for all $i\geq i_0$. Since $|\bar{w} -\bar{w}^i_{\textrm{base}}|=o(1)$, there exists an integer $j\geq i_0$ such that $\bar{w}^j_{\textrm{base}}(\tau)< \bar{w}^{j}_{\textrm{crit}}(\tau)$  with $\bar{w}^j_{\textrm{base}}(\tilde{\tau}_j) =0$ and $\bar{w}^j_{\textrm{base}}(\tau)<0$ for $\tau>\tilde{\tau}_j$, where $\tilde{\tau}_j \in I_j$ is defined by (\ref{extinct}). 
  Thus $\bar{w}(\tau)$ and $w(\tau)$ change their signs at  $\tilde{\tau}_e=\tilde{\tau}_j+o(1)$ and remain negative thereafter, thereby proving the proposition. 
\end{proof}

Note that  Propositions \ref{envelope_min} and \ref{envelope_extinct} are complementary of each other since a solution passing through a vicinity of $q_e$ either meets the conditions in Proposition \ref{envelope_min} or Proposition \ref{envelope_extinct}. To see how the method plays out,  we start by approximating $\bar{u^2}(\tau)$ by $\bar{u^2_1}_{\textrm{base}}(\tau)$ on the interval $I_1$ and examine the position of $\bar{w}(\tau)$ with respect to $\bar{w}^{1}_{\textrm{crit}}(\tau)$. If $\bar{w}(\tau)$ meets the condition in Proposition \ref{envelope_extinct}, then we obtain an early indicator of a sign change in $w(\tau)$. However, if that is not the case, we consider $I_2$ and use $\bar{u^2_2}_{\textrm{base}}(\tau)$ as a new approximation to  $\bar{u^2}(\tau)$ on the interval $I_2$. We again examine the position of $\bar{w}(\tau)$ with respect to $\bar{w}^{2}_{\textrm{crit}}(\tau)$ and check if $\bar{w}(\tau)$ meets the condition in Proposition \ref{envelope_extinct}. If not, we subsequently repeat the process and apply the technique on the  intervals $I_i, 3\leq i \leq N-k$ until $\bar{w}$ meets the  condition in Proposition \ref{envelope_extinct}. If it fails, then (\ref{cond}) must hold, in which case $\bar{w}$ attains its minimum.

\vspace{0.2in}


\section{Numerical results}
\label{sec:num_results}

\subsection{Behavior of system (\ref{nondim3})  near the  singular Hopf bifurcation} We return to  system (\ref{nondim3}) and compute the coefficients of the normal form (\ref{normal2})  near the  singular Hopf bifurcation. Treating $h$ as the input parameter and the other parameter values as in (\ref{parvalues}), 
the coefficients of the normal form (\ref{normal2}) are
\bes \label{parvalues2}
\delta \approx 0.2504,\ F_{13} \approx 0.1173,\ F_{111} \approx -0.8663,\ H_3\approx 0.0377,\  H_{11}\approx -0.1691,
\ees
with $\alpha(h)= -145.8265h+38.589$ being the varying parameter. 
The eigenvalues of the variational matrix of system (\ref{normal2}) at the equilibrium $q_e=(0, 0, 0)$ (which corresponds to the coexistence equilibrium $E^*$ of system (\ref{nondim3})) are $\lambda_1 \approx 0.00943$ and a complex pair $\lambda_{2,3}$ with  $\textnormal{Re}(\lambda_{2,3})<0$ if $-3.997<\alpha<0$.  
A subcritical  Hopf bifurcation occurs at  $\alpha=\alpha_H=0$  (which corresponds to  $H_2\approx 0.2646$ in system (\ref{nondim3})), where the first Lyapunov coefficient $l_1(0) \approx 0.0934>0$.  In an immediate neighborhood of the Hopf bifurcation, system (\ref{normal2})  undergoes a subcritical torus bifurcation at $\alpha=-0.011$ (corresponding to  $TR \approx 0.2647$ in system (\ref{nondim3})), stabilizing the family of cycles $\Gamma^{\alpha}$ born at $\alpha=0$. 
  
   \begin{figure}[h!]     
  \centering 
   \subfloat[IC = $(0.452,0.432,0.329)$]{\includegraphics[width=6.4cm]{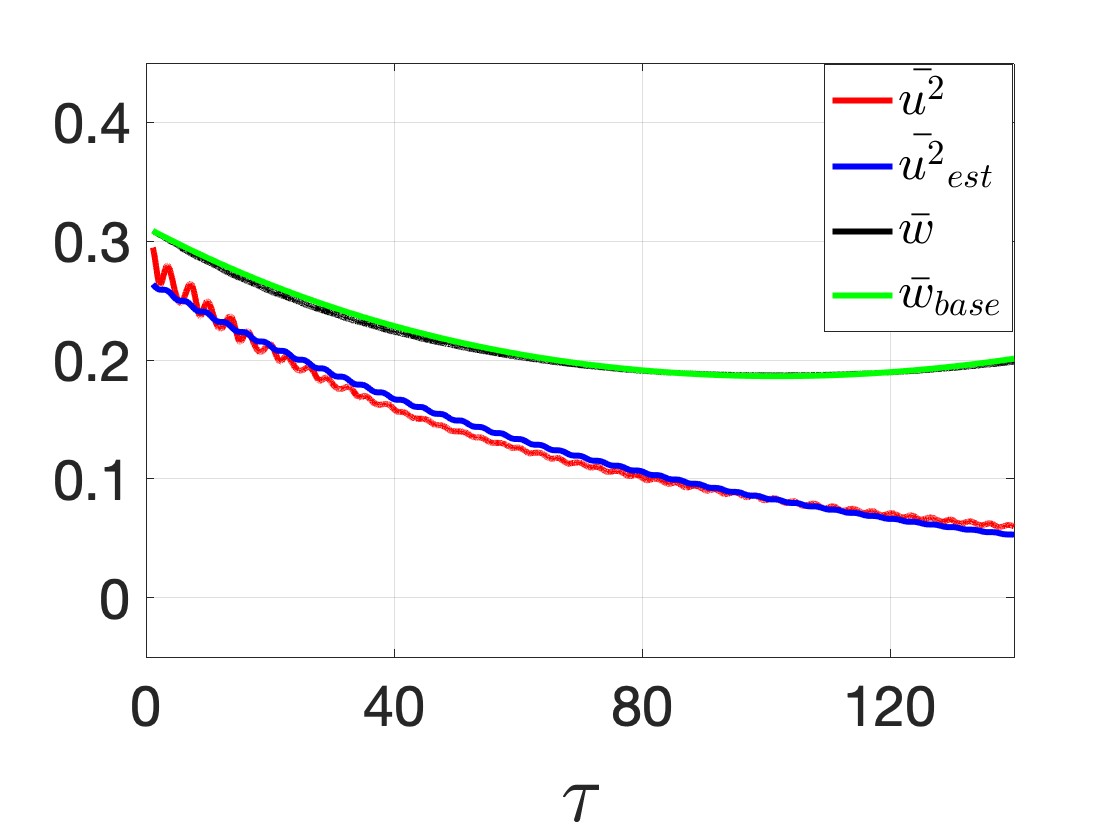}}
    \subfloat[IC=$(0.452,0.432,0.259)$]{\includegraphics[width=6.4cm]{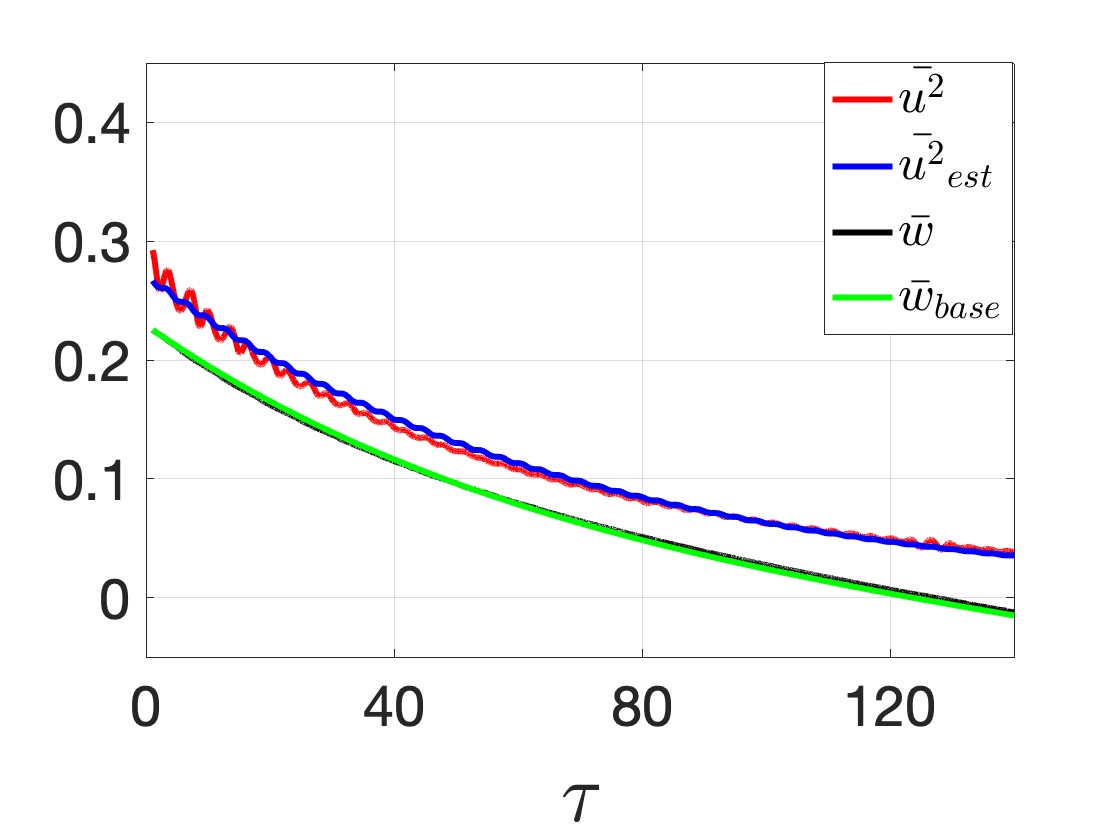}}\quad
  \caption{Graphs of $\bar{u^2}(\tau)$,  $\bar{w}(\tau)$ and their approximations by $\bar{u^2}_{\textrm{est}}(\tau)$ and $\bar{w}_{\textrm{base}}(\tau)$ respectively for system (\ref{normal2}) with parameter values as in (\ref{parvalues2}) and $\alpha=-0.04$}
  \label{moving_averages}
  \end{figure}
  
        \begin{figure}[h!]     
  \centering 
    \subfloat[]{\label{moving_averages_threshold_a}{\includegraphics[width=6.4cm]{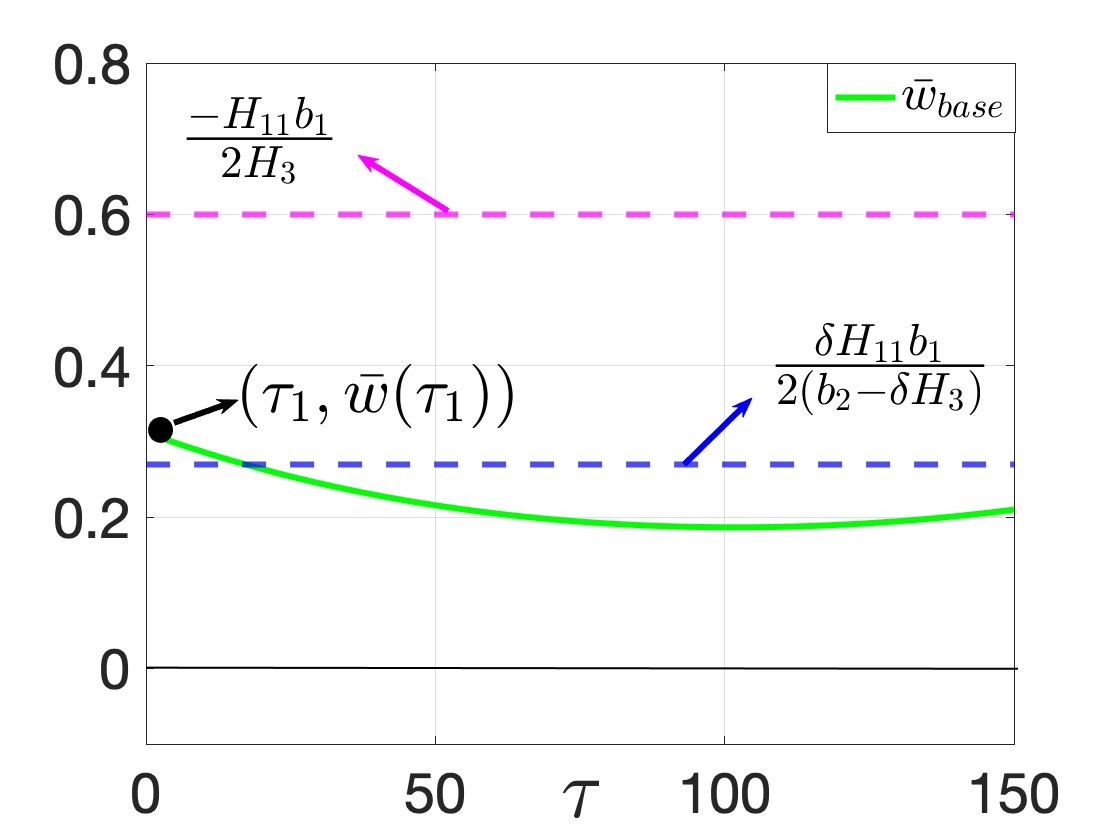}}}
    \subfloat[]{\label{moving_averages_threshold_b}{\includegraphics[width=6.4cm]{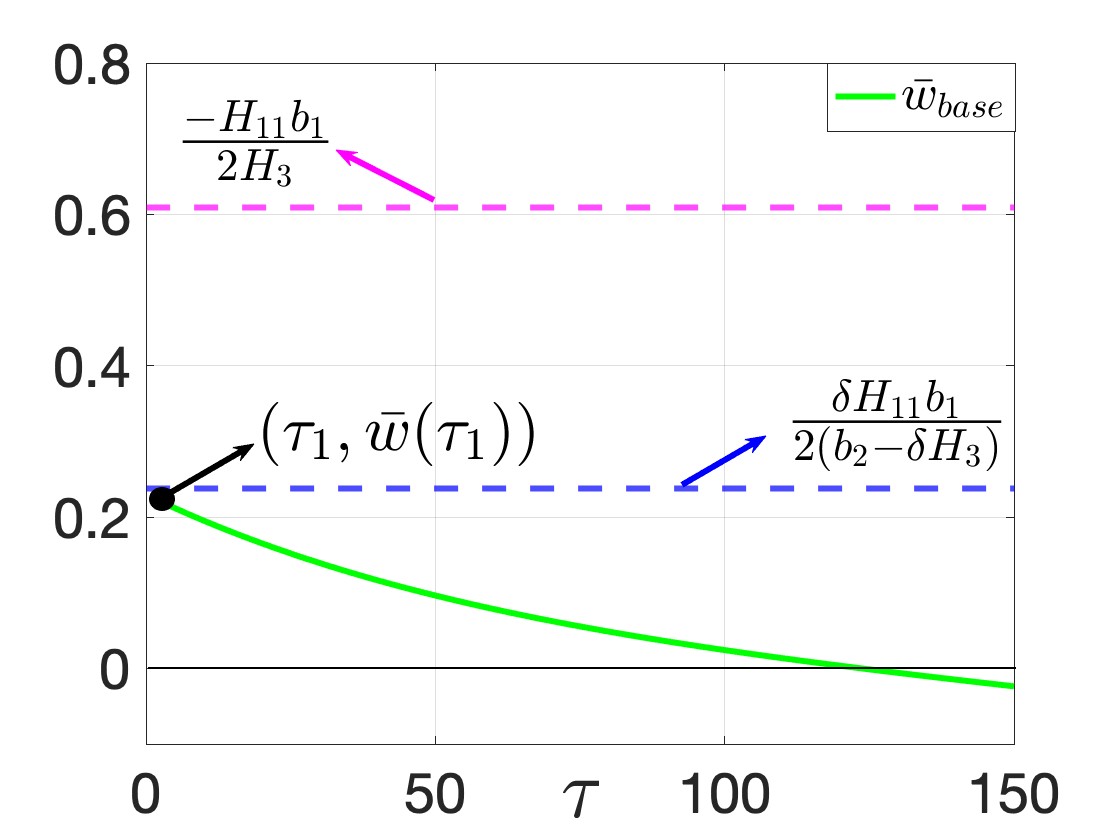}}}
  \caption{Behavior of $\bar{w}_{\textrm{base}}(\tau)$ for two different initial values as in figure \ref{moving_averages}. In both cases, $N=18$ oscillations were chosen. (a) Here  $\tau_1=1.094$, $\tau_N =109.2$,  $b_1=0.2799$ and $b_2=-0.0128$. (b) Here $\tau_1=1.0784$,  $\tau_N =109.04$,  $b_1=0.2799$ and $b_2=-0.0156$.}
  \label{moving_averages_threshold}
  \end{figure}
  
  In a vicinity of $\alpha=-0.011$, system (\ref{normal2}) exhibits bistability between $\Gamma^{\alpha}$ and $(0, 0, -\infty)$. 
  Note that $|\alpha|=O(H_3)$ here. In both cases, $u(\tau)$ and $v(\tau)$ exhibit very similar oscillatory patterns near $q_e$, however,  $w(\tau)$ attains a global minimum in one case, while in the other, it fails to attain a minimum and becomes negative. 
 In figure \ref{moving_averages}, we plot the time series of $\bar{u^2}(\tau)$, $\bar{w}(\tau)$ along with the approximating functions  $\bar{u^2}_{\textrm{est}}(\tau)$ and $\bar{w}_{\textrm{base}}(\tau)$ respectively, where  $\bar{u^2}_{\textrm{est}}(\tau)$  and $\bar{w}_{\textrm{base}}(\tau)$ are defined by (\ref{uest}) and (\ref{west1}) respectively. Note that the approximating curves lie very close to the actual curves uniformly over the time interval under consideration. By Theorem \ref{bistability}, $\bar{w}_{\textrm{base}}(\tau)$ attains a minimum if it satisfies condition (\ref{condnew}) as shown in figure \ref{moving_averages_threshold_a}. On the other hand, if $\bar{w}(\tau_1)<\frac{\delta H_{11} b_1 e^{b_2 \tau_1}}{2(b_2-\delta H_3)}$, $\bar{w}_{\textrm{base}}(\tau)$ eventually becomes negative as shown in figure \ref{moving_averages_threshold_b}.



\subsection{Analysis of the time series in  figure \ref{timeseries_all}} 

Finally, we consider the time series in figure \ref{timeseries_all}. We first recall that the initial conditions for both solutions lie in a close neighborhood of $W^s(E^*)$, the stable manifold of $E^*$, and are in $\Sigma^+$, where $\Sigma^+$ is the upper-half space containing the limit cycle $\Gamma$ of system (\ref{nondim3}). 
We next note that the tangent plane $\Sigma$ in system  (\ref{nondim3}) corresponds to $\{w=0\}$ in system (\ref{normal2}); the limit cycle $\Gamma$ and the boundary equilibrium point $E_{xz}$ lie on the opposite sides of $\{w=0\}$. We now use the transformations in (\ref{convert}) to map $(x, y, z)$ to $(u, v, w)$ and plot the solutions in their normal form with respect to the slow time variable $s$,  as shown in  
figure \ref{transformed_normal}.
 \begin{figure}[h!]     
  \centering 
  \subfloat[]{\includegraphics[width=6.1cm]{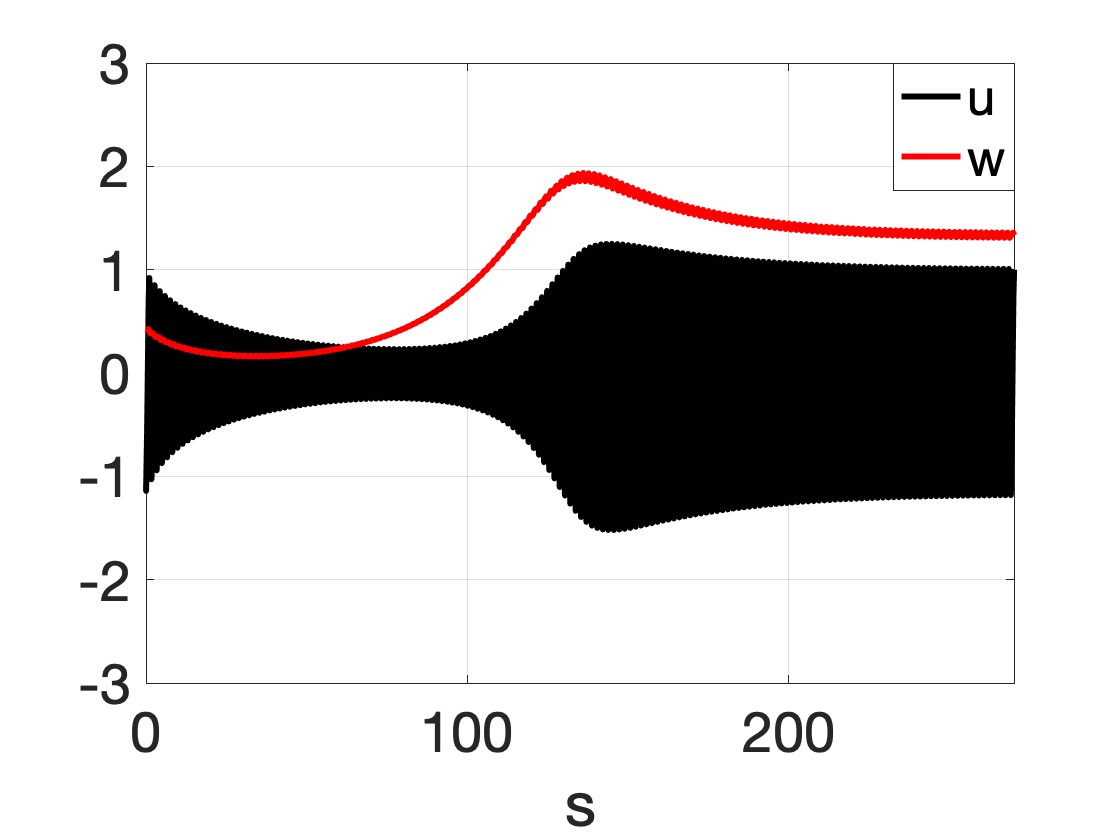}}
  \quad
     \subfloat[]{\includegraphics[width=6.1cm]{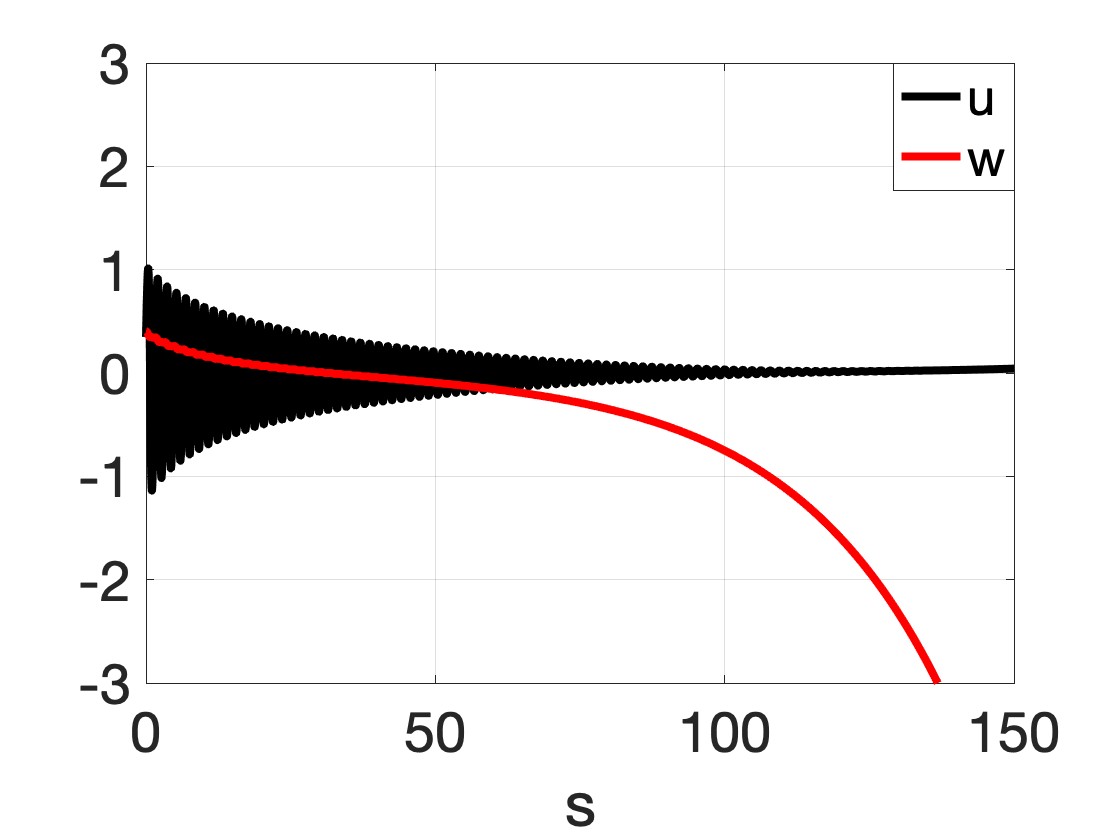}}
 \caption{Time series of $x(s)$, $y(s)$ and $z(s)$ in figure \ref{timeseries_all}  in terms of the normal form variables $u(s)$ and $w(s)$. The left panel corresponds to figure \ref{timeseries_all_a}  and the right panel to figure \ref{timeseries_all_b}. }
 \label{transformed_normal}
\end{figure}
 Below, we write down the slow variable $w$ in terms of $x, y$, and $z$ using (\ref{convert}), which will be used in monitoring early warning signals:
\[ w =  - 0.286x+ 689.716y - 207.413 z+ 5.315.
\] We also draw $\Sigma$ and the inverse image of  
  the funnel  $\mathbb{R}^3\setminus \Omega_{\alpha}$ in the $xyz$-coordinates with  the phase spaces of the trajectories corresponding to figure \ref{timeseries_all} superimposed on them  as shown in   figure \ref{bistable_phasespace}. The initial transient phases of the trajectories were shown in figure \ref{initial_transients}.  As expected from the analysis of  system (\ref{normal2}), figure \ref{bistable_phasespace} shows that the solution associated with figure \ref{timeseries_all_a} enters into the funnel and approaches $\Gamma$, while the solution corresponding to figure \ref{timeseries_all_b}  always remains outside the funnel, and eventually approaches the boundary equilibrium state $E_{xz}$.

\begin{figure}[h!]     
  \centering 
    \subfloat[]{\includegraphics[width=6.3cm]{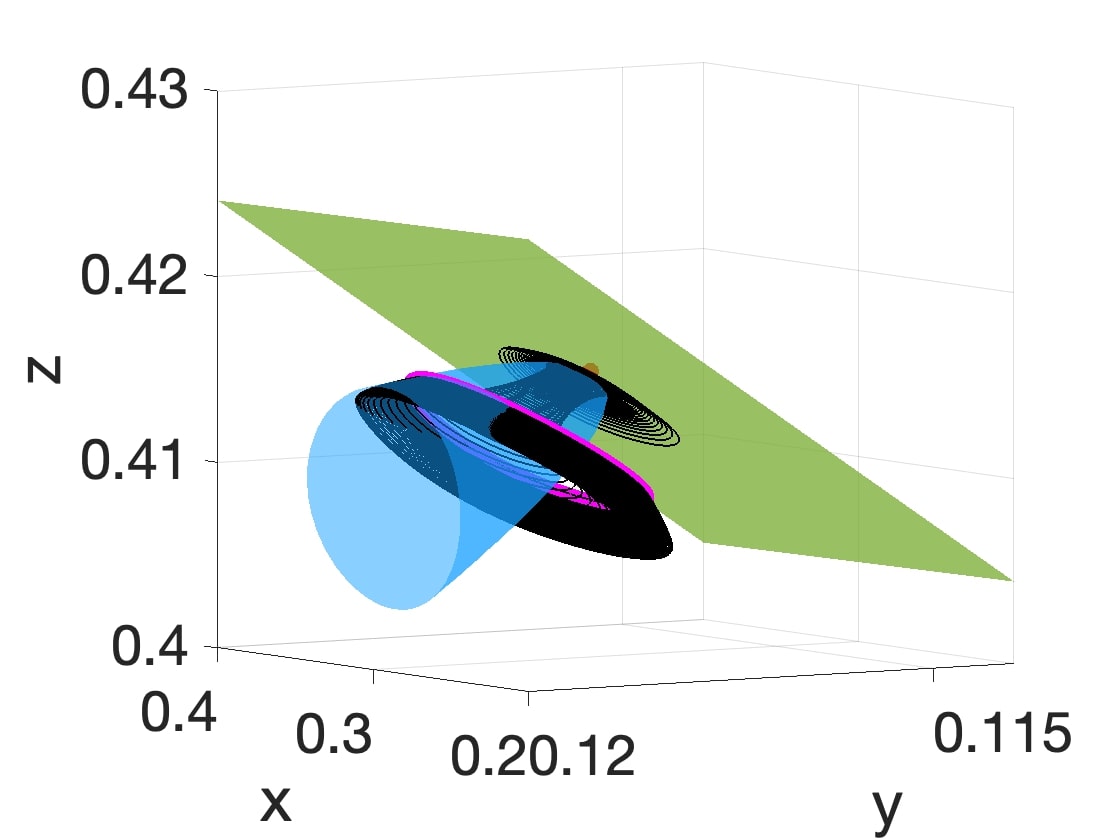}}
  \subfloat[]{\includegraphics[width=6.3cm]{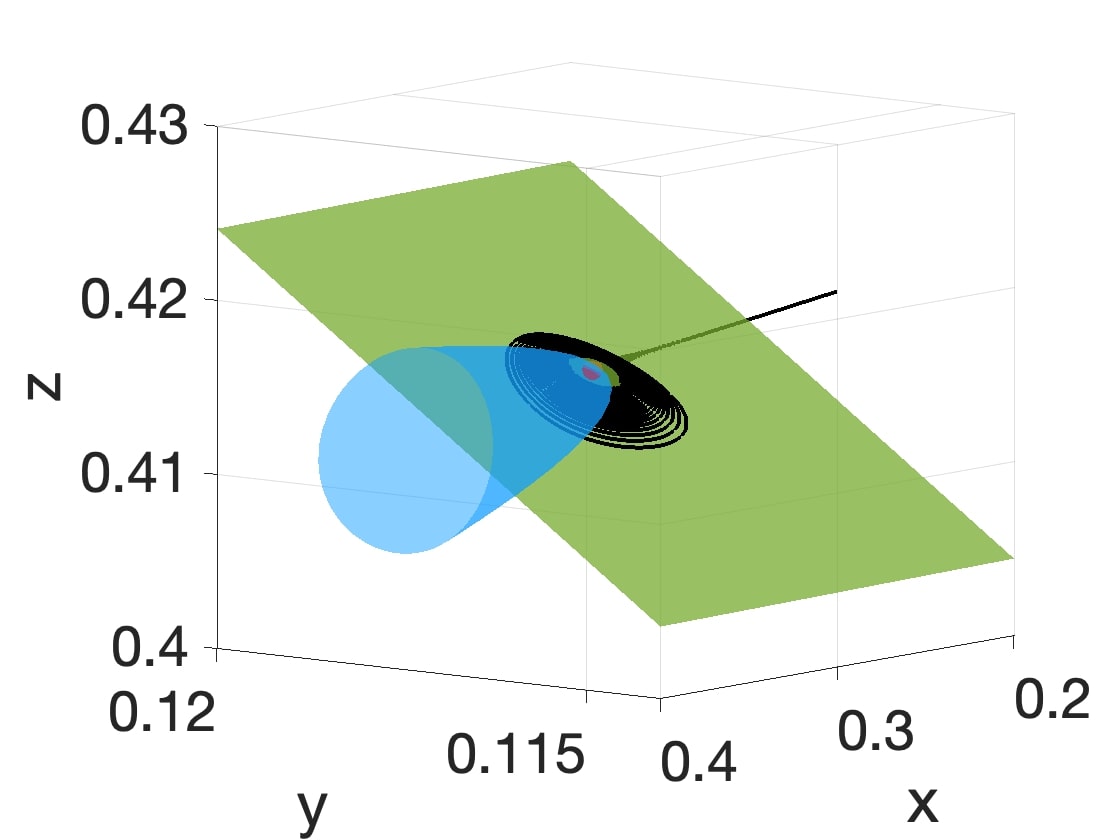}}\qquad
 \caption{Phase spaces of the solutions in  figure \ref{timeseries_all} approaching $\Gamma$ (in magenta) or  $E_{xz}=(0.357, 0, 0.615)$.  Also shown are $E^*$ (red dot),  $\Sigma$ (olive green plane), and the inverse image of the funnel $\mathbb{R}^3\setminus \Omega_{\alpha}$ (deep sky blue surface). (A): Initial conditions: $(0.2785,0.1181, 0.4164) \in \Sigma^+$. The solution enters into the funnel to approach $\Gamma$. (B): Initial conditions: $(0.2831,0.1184, 0.417) \in \Sigma^+$. The solution fails to enter the funnel and approaches  $E_{xz}$.}
 \label{bistable_phasespace}
\end{figure}

We next use the method discussed in Section \ref{sec:early_warning} to detect an early warning signal of an impending transition in the population of $y$.  To this end, we consider the time series in  the normal form variables shown in figure \ref{transformed_normal} 
and  construct intervals $I_i=[s_1, s_{k+i}]$, $i=1, 2, \ldots N-k$ with  $N=41$ and $k=5$, where $s_i$ are the locations of relative maxima of $u(s)$  such that $\{u(s_i)\}$ is decreasing. The choices for $N$ and $k$ are flexible. The value of $N$ was chosen to ensure that the amplitude of oscillations of $u$ decayed at least by a factor of $e$ when $s=s_N$ as stated in Section \ref{sec:early_warning}. Similarly, the value of $k$, which represents an optimal number of oscillations required for a good approximation of $\bar{u}$ in the interval $I_1$, was chosen to be five. Using a least square data fitting tool, we numerically approximate $\bar{u^2_i}_{\textrm{base}}(s)= k^i_1 e^{k^i_2(s-s_1)}$ and obtain that the amplitudes $\{k^i_1\}_{i=1}^{36}$ and the decay constants $\{k^i_2\}_{i=1}^{36}$ are monotonically decreasing and increasing respectively for both solutions. Corresponding to figure \ref{transformed_normal}(a), we obtain $k^1_1 =0.547$, $k^{36}_ 1=0.4643$,   $k^1_2=-0.1151$,  $k^{36}_2 =-0.0627$ with $s_1 =0.27$, while that for figure  \ref{transformed_normal}(b), we obtain $k^1_1=0.5489$, $k^{36}_1 =0.5089$,   $k^1_2 =-0.0995$, $k^{36}_2 =-0.0757$  with $s_1 =0.375$.  Note that $k^i_1=O(\delta) \approx 0.25$ and $k^i_2=O(\alpha \delta)\approx -0.1$. 

\begin{figure}[h!]     
  \centering 
  \subfloat[]{\label{early_warn_a}{\includegraphics[width=6.3cm]{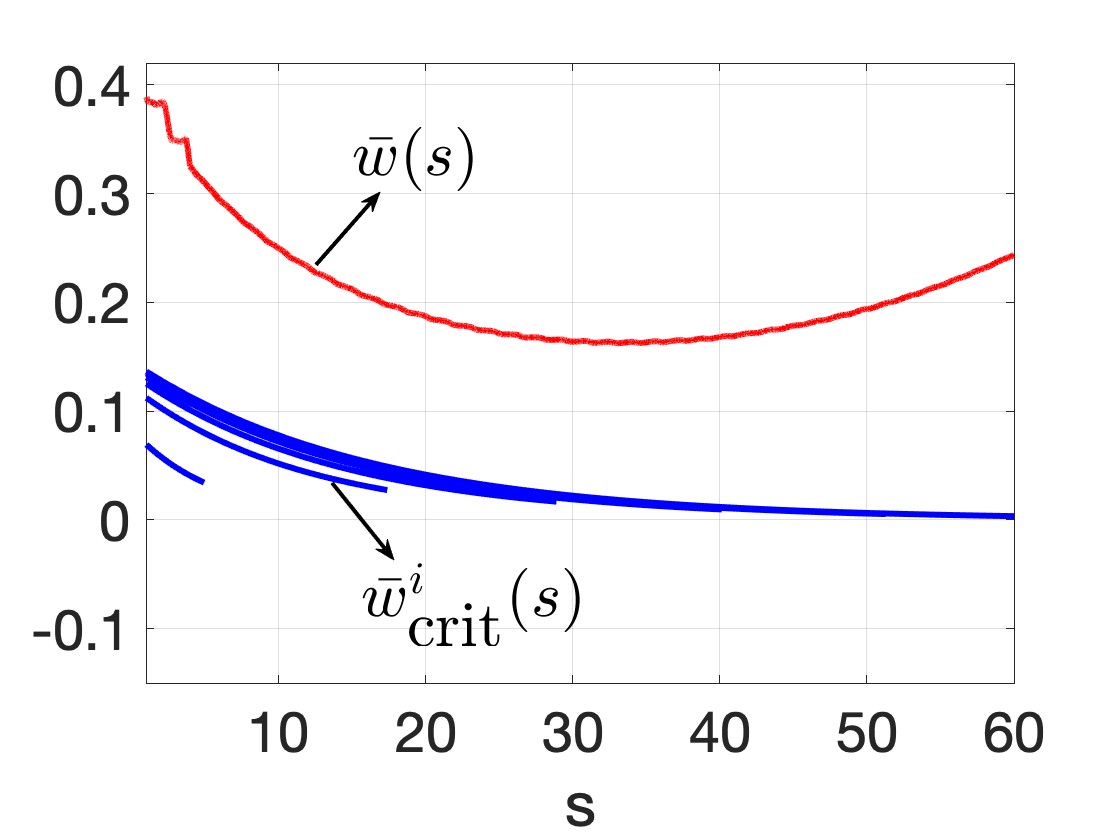}}}
     \subfloat[]{\label{early_warn_b}{\includegraphics[width=6.3cm]{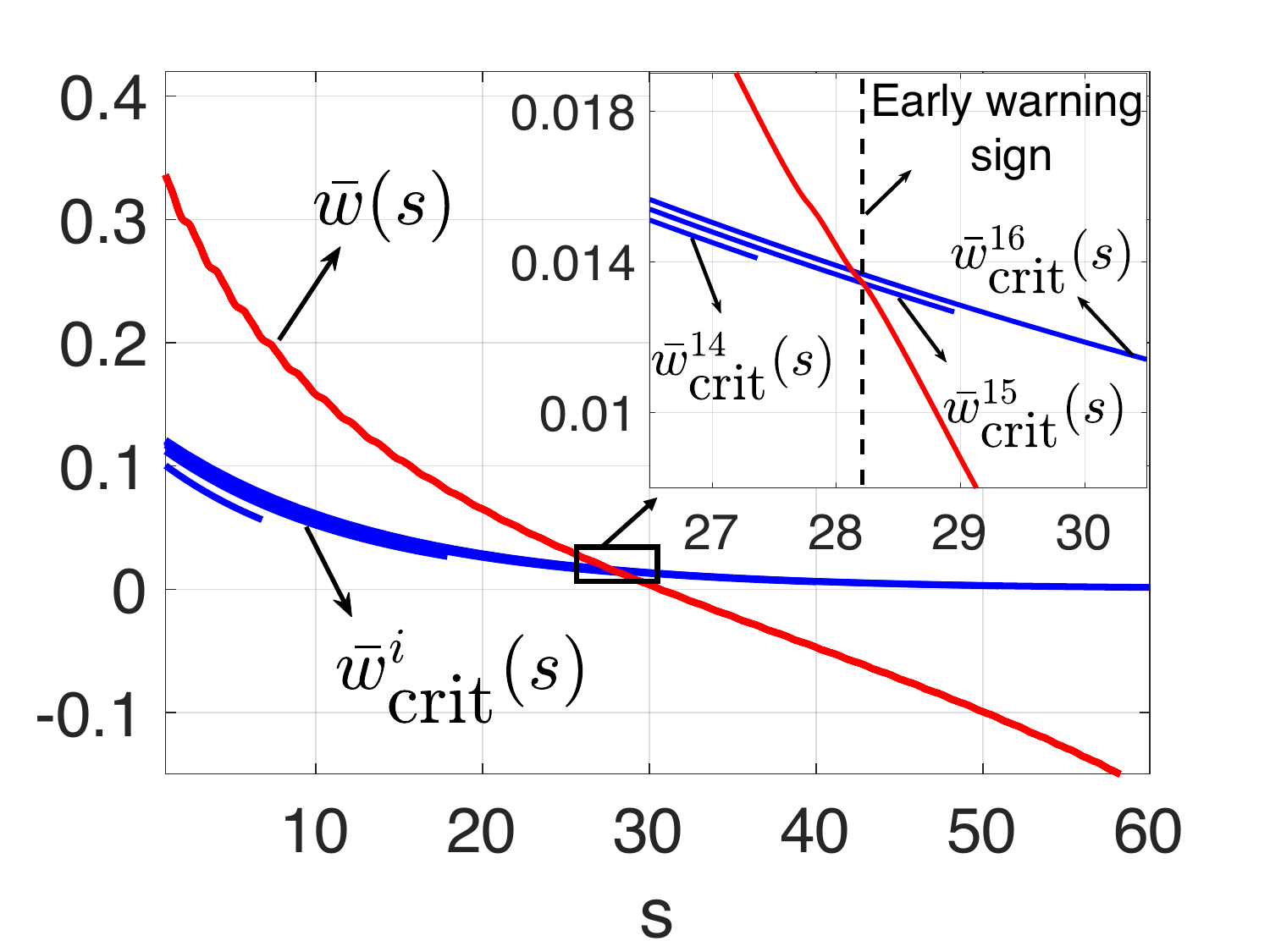}}}
     \quad
    \subfloat[]{\label{early_warn_c}{\includegraphics[width=6.1cm]{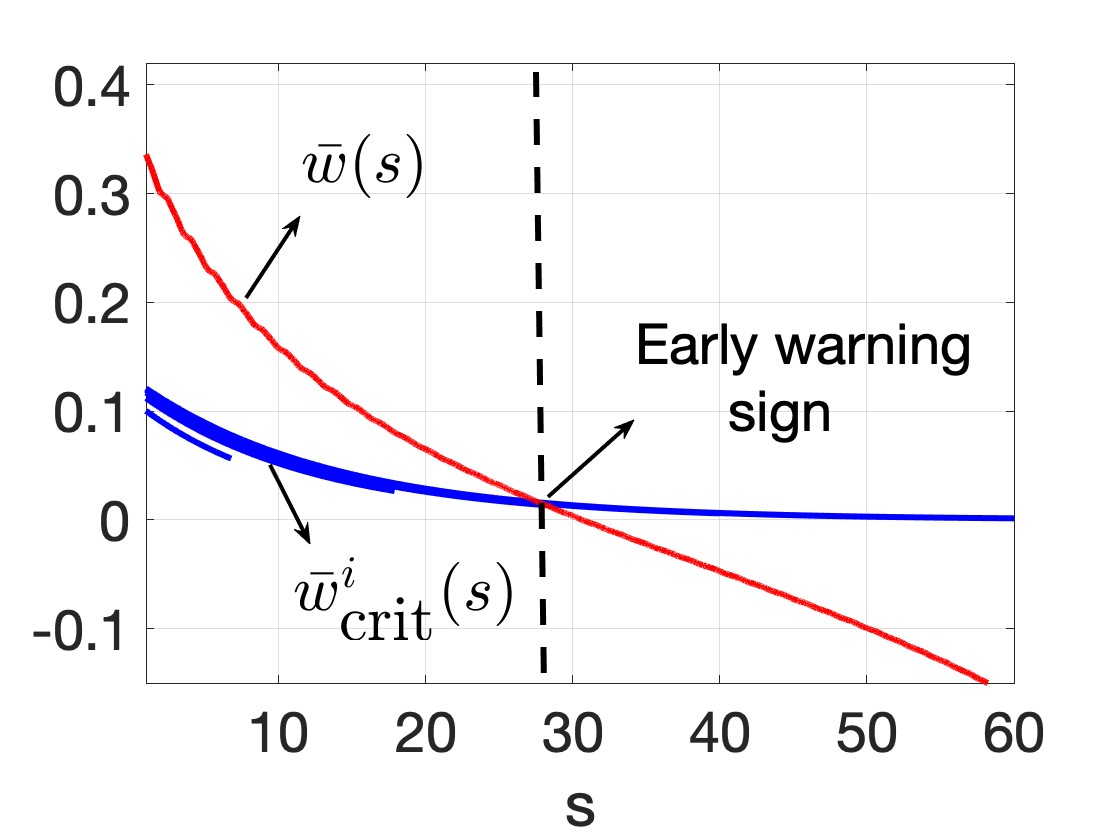}}}
     \subfloat[]{\label{early_warn_d}{\includegraphics[width=6.1cm]{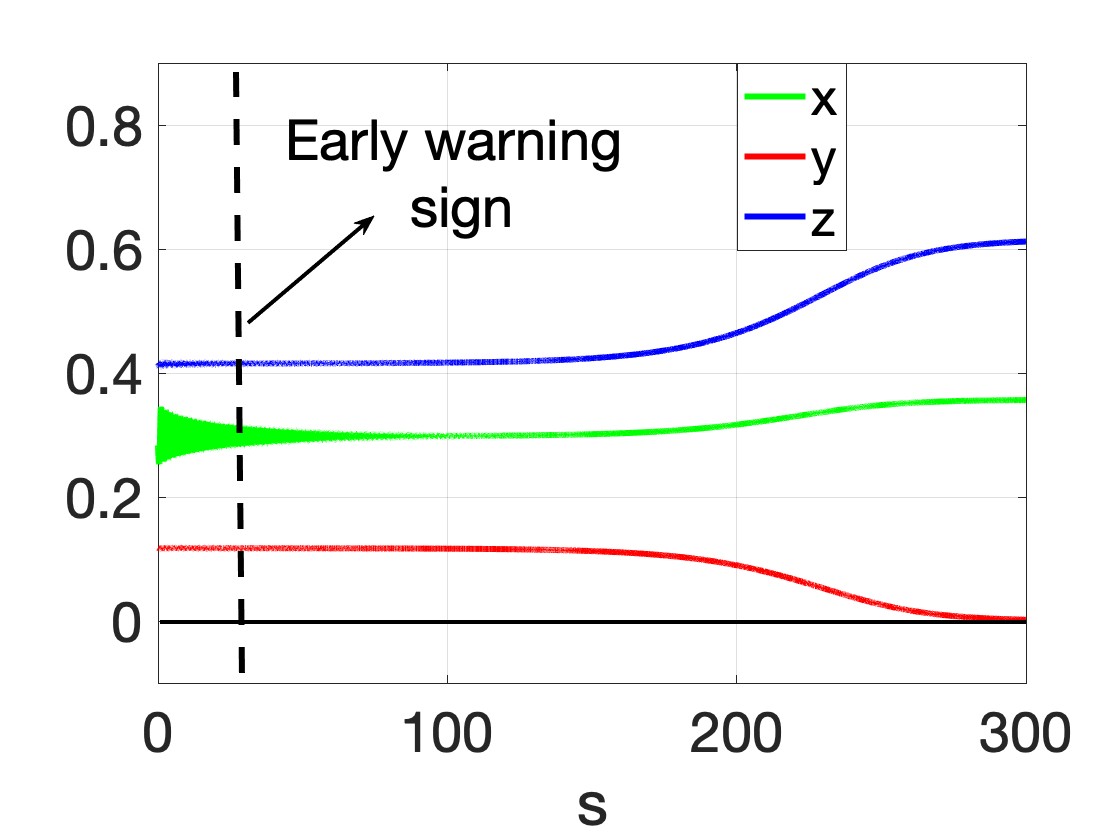}}}   
 \caption{Detection of an early warning sign of species' extinction in system (\ref{nondim3}) for initial conditions as in figure \ref{timeseries_all}.  (a) All species coexist as $\bar{w}(s)$ attains a global minimum. (b) One of the species will eventually extinct as predicted by the crossing of $\bar{w}$ with $\bar{w}^{15}_{\textrm{crit}}$. (c)-(d) Earliest warning sign of extinction of $y$ is detected at $s\approx 28.3$ when $\bar{w}(s)$ crosses $\bar{w}^{15}_{\textrm{crit}}(s)$.}
 \label{application_normal}
\end{figure}

We now find the critical curves $\{\bar{w}^{i}_{\textrm{crit}}(s)\}$, $1\leq i \leq 36$,  defined by (\ref{crit_curve}), where $s=\delta \tau$, and  monitor the position of  $\bar{w}(s)$ with respect to $\bar{w}^{i}_{\textrm{crit}}(s)$ on $I_i= [s_1,s_{i+k}]$ for each $i$.  
It turns out that $\bar{w}(s)$ corresponding to figure \ref{transformed_normal}(a) lies above  $\bar{w}^{i}_{\textrm{crit}}(s)$ for all $1\leq i\leq 36$, and therefore by  Proposition \ref{envelope_min}, it attains its minimum  (figure \ref{early_warn_a}). On the other hand,  $\bar{w}(s)$ associated with figure \ref{transformed_normal}(b) crosses  $\bar{w}^{15}_{\textrm{crit}}(s)$ at $s={\tilde{s}}_{15} \approx 28.3$  and remains below $\bar{w}^{i}_{\textrm{crit}}(s)$ for $ s> {\tilde{s}}_{15}$  and all $i\geq 15$, and therefore by Proposition  \ref{envelope_extinct},  $\bar{w}(s)$ becomes negative for some $s> {\tilde{s}}_{15}$ (figures \ref{early_warn_b}-\ref{early_warn_c}). 
Since $E_{xz}$ is the other attractor that lies below the $\{w=0\}$ plane, the change in sign of $w(s)$ gives a signal of an impending transition in the population density of $y$.

We also remark that the time series of $x$, $y$ and $z$ in figure \ref{early_warn_d} do not show any significant changes in the population trends until $s\approx 200$, after which the population of $y$  changes abruptly. At this point, it could be too late for an intervention to prevent $y$ from extinction. Our analysis allows us to detect the impending regime shift at $s={\tilde{s}}_{15}$, significantly in advance of the transition, thus giving access to an early intervention.


\section{Summary and outlook}
\label{sec:conclusion}

 Detection of early warning signals of an impending population shift  is challenging mainly since the underlying mechanisms for regime shifts are not always known. Most ecosystems are usually resilient to slow changes in environmental parameters, showing little change in their state prior to a transition \cite{biggs}. Though statistical indicators of critical transitions have been demonstrated in many systems, studies have shown that they do not necessarily apply for anticipating all types of regime shifts \cite{dakos}. Hence, analysis of long transients  and development of relevant mathematical methods and tools  can offer an alternate avenue for efficient forecasting \cite{morozovetal}.  In this paper, we took an approach to predict a regime shift and establish an early recognition of the impending transition by exploiting the inherent timescale separation in the oscillatory patterns of long transient dynamics in a class of three-dimensional predator-prey models.

We focused on a class of slow-fast predator-prey models that studies interaction between two species of predators competing for their common prey with asymmetric competition between them, and exhibits a bistable behavior in a parameter regime near a subcritical singular Hopf bifurcation. We considered one of the attractors to be  periodic that resulted from a Neimark-Sacker  bifurcation of the saddle cycle born from the subcritical singular Hopf bifurcation of a coexistence equilibrium point, while the other could be either a point attractor or a periodic attractor that doesn't necessarily bifurcate from the singular Hopf point. We assumed that the two attractors lie on opposite sides of the tangent plane of the stable manifold of the coexistence equilibrium point.  
The coexistence equilibrium is saddle-focus with one-dimensional unstable manifold such that one of its branches  tends to one of the attractors while the other tends to the other attractor. We further assumed that the rate of expansion in the axial direction does not exceed the rate of radial contraction in a vicinity of the equilibrium, both rates being very small.  With these assumptions, we took a systematic rigorous approach to unravel the dynamics near such an equilibrium and predict asymptotic behaviors of solutions exhibiting nearly indistinguishable transient dynamics in a vicinity of it.  In particular, we utilized the geometry of the invariant manifolds, the underlying separation of timescales, and the bifurcation structure to develop a suitable method for analyzing the transients.

The local bifurcation analysis in the vicinity of the singular Hopf bifurcation shed light on the mechanism underlying the intrinsic nature of the observed transients, characterized by their rapid oscillations and slow variation in the amplitude.  We noted that the rapid oscillations in the transients can be attributed to  the singular Hopf bifurcation of the saddle-focus equilibrium,  while the slow radial contraction and slow axial expansion along the invariant manifolds of the saddle-focus, another distinguishing trait of the transients studied in this paper,  were caused by the proximity of the system to a zero-Hopf bifurcation (see figure \ref{two_par_bif} for system (\ref{nondim3})). The combined effect of the two bifurcations (which could also be termed as a singular zero-Hopf bifurcation) gave rise to the long transients in system (\ref{normal1}). To the best of our knowledge, such a mechanism underlying long transients in higher-dimensional models has not been explored yet. We addressed the difference between the timescales of the high-frequency oscillations and the slow drift rate in the amplitude by reducing the system to a suitable normal form, so that the rates of expansion and contraction towards the saddle-focus equilibrium could be expressed in terms of the distance of the bifurcation parameter, $\alpha(h)$, from the subcritical Hopf bifurcation and the coefficient $H_3$ of the linear term in the equation of the slow variable $w$ of the normal form. Furthermore, treating the slow variable as a bifurcation parameter, we gained insight into the signs of the coefficients of the equations of the fast variables in the normal form, which was crucial in analyzing the bistable behavior.  

A key component of our study included the derivation of an analytical expression for $\bar{w}$, the averaged dynamics of the slow variation along the unstable manifold of the saddle-focus equilibrium using a moving average technique. We noted that $\bar{w}$ can be expressed as a linear combination of an exponential decay curve and an exponential growth curve, and its long term behavior is determined by the competing effect of the coefficients of the exponential curves. The analysis then led us in finding a set of sufficient conditions for predicting asymptotic behaviors of solutions that start in the vicinity of the saddle-focus point. The analysis was further extended to devise a method for finding an early warning signal of an impending population collapse. The two qualitative different behaviors of $\bar{w}$ provided insight into defining a sequence of critical curves $\{\bar{w}^{i}_{\textrm{crit}}\}_{i=1}^N$, which potentially separate solutions approaching the periodic attractor from solutions approaching the other attractor, and their positions  with respect to the slow variable facilitated obtaining early indicators of a forthcoming collapse.  Our analysis suggests that a successful early warning signal can be obtained by monitoring the behavior of the slow variable $w$, which can be expressed as a suitable linear combination of the  state variables $x, y$, and $z$, near the saddle-focus equilibrium (see the change in coordinates transformation in  (\ref{convert})). Our analysis sheds light on the ability to not only distinguish between time series with contrasting long-term dynamics despite presenting similar trends at the onset (see figure \ref{timeseries_all}), but also on the power to detect warning signals of catastrophic population collapses remarkably ahead of time as shown in figure \ref{application_normal}.

The geometric treatment of system (\ref{normal1}) via the normal form coordinates addresses the subtle nature of the apparently alike  dynamics near the saddle-focus equilibrium.  Rectifying the local unstable manifold of the saddle-focus equilibrium along one of the coordinates axes of the transformed system, namely the $w$-axis, and mapping the principal directions of the local stable manifold to the other two coordinate axes, the $u$-$v$ axes, using the transformations in \cite{BB}, enabled us in viewing and representing the ``near cylindrical symmetry" of the dynamics near the saddle-focus in terms of the normal form variables. In addition to visualizing the locally invariant manifolds, we get a closer view of the geometry of the basin of attraction of the periodic attractor in the vicinity of the saddle-focus equilibrium. We obtain an analytical expression of the graph of the basin boundary of the two attractors, namely the surface of $\Omega_{\alpha}$, and observe that it  partitions the half-space lying above the stable eigenplane into two sets. This  geometrical perspective subsequently guided us in explaining  the dynamics inside the volume enclosed by the locally stable manifold while elucidating the concurrence of long-term contrasting dynamics as solutions experience a slow passage through the saddle-focus.

Near the singular Hopf point, we remark that the time to approach an asymptotically stable attractor diverges as the initial conditions approach the  basin boundary. 
 In general, such quasi-stationary dynamics or ``crawl-by" behavior  occur when a system spends a significant amount of time near a saddle-type invariant set as it approaches the invariant set along the stable manifold  \cite{Hastings2, rubinetal}.  The classical example  of a system displaying such dynamics is the  Lotka-Volterra model of competition between two species for limited resources \cite{murray}, which can exhibit bistability between the two competitive exclusion states of the species (i.e. one species is the sole survivor) whose basins of attraction are separated by the stable manifold of the saddle coexistence equilibrium.  
 Our work extends this idea to higher dimensions.  We further remark that the effect of saddle-type invariant sets leading to long transients have been also observed in several field and laboratory studies \cite{Hastings2, morozovetal}. One such example includes the bistable dynamics observed in  laboratory yeast populations as a feedback between population and evolutionary dynamics \cite{gore_nature, gore_plos}. The system displays damped oscillatory behavior characterized by two timescales as it approaches a tipping point (fig. 5 in supplementary material of \cite{gore_nature}), similar to the transient oscillatory dynamics seen in system (\ref{nondim3}).

 The analysis performed in this paper can be applied to other types of predator-prey models. One such example may include intraguild food web models studying the interaction between a class of intraguild predators,  a class of intermediate predators, and their shared prey. Such models have been documented to exhibit bistability between periodic and point attractors (see Fig. 10 in \cite{kang}) or between two point attractors (see Fig. 11 in \cite{kang}), where the invariant manifolds of a nearby saddle-focus equilibrium seem to  organize the dynamics. Depending on the choice of the functional responses of the predators, the coexistence equilibrium point of such systems may exhibit zero-Hopf bifurcation and a branch of torus bifurcations can emanate from the zero-Hopf point \cite{BL}. It remains to be checked if the models considered in \cite{BL} meet conditions (Q1)-(Q5).  We also remark that our analysis can be extended beyond predator-prey models. Systems with similar bifurcation structures and geometric configurations having multiple coexisting attractors whose basins of attraction are separated by the invariant manifolds of a saddle-focus equilibrium can be analyzed using the techniques used in this paper. The technique for analyzing the transient dynamics will be applicable to any time series that features fast oscillations with a slowly varying amplitude. Ecological data that display slowly damped oscillatory behavior with an underlying timescale separation \cite{gore_nature} can be studied using our method by formulating suitable data-driven models which may fall within the class of generic predator-prey models considered in this paper. The method developed in this work can thus be viewed as an effort to understand the behavior much in advance of a critical transition in systems with nonlinear interaction between multiple species, which can presumably be helpful to understand critical transitions in complex ecological networks.

The dynamics of system (\ref{nondim3}) are associated with survivability and coexistence of species and thereby bear connections to the phenomena of permanence and persistence in ecological communities \cite{cantrell}. It is known that random variations in births or deaths or changes in environmental conditions can shift a population from the basin of attraction of one attractor to the other resulting in uncertainties in prediction of long-term behaviors in an ecosystem. Hence, a key question is how does our analysis contribute to finding early warning indicators of population extinction in a stochastic setting? The answer to this question lies in understanding the deterministic skeleton and performing a careful analysis of the crawl-by transients exhibited near the saddle-focus equilibrium.  An initial numerical investigation shows that for several sample paths, the slow passage through the saddle-focus equilibrium remains preserved in presence of uncorrelated multiplicative noise (which is chosen to be much smaller compared to the timescale separation) as the paths remain concentrated in neighborhoods of the deterministic solution, where the size of the neighborhood is proportional to the strength of the noise. Depending on the strength and properties of the noise, our method of finding early warning signals of extinction can be possibly extended, and a multi-scale approach as in \cite{kuskeetal} can be taken to derive the envelope equation for the stochastic amplitude of the oscillations. However, a detailed analysis is beyond the scope of this paper. We also leave as future work to study the effect of noise on the length of the transient dynamics \cite{Hastings3, reimer}.

We finally remark that our averaging approach can be applied to analyze a wide range of oscillatory dynamics characterized by two intrinsic timescales that may extend beyond autonomous systems.  Additional avenues for future research  include expanding the analysis to wider range of parameters such as to parameter regimes where amplitude-modulated oscillations past a torus bifurcation are observed as documented in \cite{Sadhunew}, or in regimes where long transients in the form of mixed-mode oscillations preceding a population collapse occur (such as in the region immediately past $PD_1$ in  figure \ref{one_par_bif}), and extending the system to time-dependent parameters.

\section*{Appendix}
\label{sec:appendix}
  \emph{Expressions for the coefficients in system (\ref{normal2}):} The expressions for $\omega$, $F_{13}$, $F_{111}$, $H_3$, $H_{11}$, and $\alpha(h)$ in Theorem \ref{normal} are given below:
 \begin{eqnarray*} 
\left\{
\begin{array}{ll}
\omega &=\sqrt{-\bar{x} (\bar{y}\bar{\phi_y}\bar{\chi_x}+\bar{z} \bar{\phi_z}\bar{\psi_x})},\\ [5pt]

F_{13}&=  \bar{x} \bar{\phi_z}(\bar{z}\bar{\psi_z}-\bar{y}\bar{\chi_y})+\frac{\bar{x}}{\bar{\phi_y}}(\bar{y}\bar{\phi^2_y}\bar{\chi_z} -\bar{z}\bar{\phi^2_z}\bar{\psi_y}) +\frac{\omega^2}{\bar{\phi_{xx}}\bar{\phi_y}}(\bar{\phi_{xz}}\bar{\phi_y} - \bar{\phi_{xy}}\bar{\phi_z}),\\[5pt]

F_{111} &= \frac{\omega^2}{\bar{x}^2\bar{\phi^2_{xx}}} (3\bar{\phi}_{xx}+\bar{x}\bar{\phi}_{xxx}) +\frac{\bar{x}}{\omega^2}[\bar{y}\bar{\chi_x}(\bar{y} \bar{\phi_y}\bar{\chi_y} + \bar{z}\bar{\phi_z}\bar{\psi_y})+ \bar{z}\bar{\psi_x} (\bar{y}\bar{\phi_y}\bar{\chi_z}+\bar{z} \bar{\phi_z}\bar{\psi_z} )] \\
&+\frac{1}{\bar{x}\bar{\phi}_{xx}}[\bar{y}\bar{\chi_x}(\bar{\phi_y}+\bar{x}\bar{\phi}_{xy}) +\bar{z}\bar{\psi_x}(\bar{\phi_z} +\bar{x}\bar{\phi}_{xz})+ \bar{x}(\bar{y}\bar{\phi_y}\bar{\chi}_{xx}+ \bar{z}\bar{\phi_z}\bar{\psi}_{xx})], \\ [7pt]

 H_3 &= \frac{\bar{x}\bar{y} \bar{z}}{\omega^2}[\bar{\psi_x}(\bar{\phi_y} \bar{\chi_z} -\bar{\phi_z}\bar{\chi_y} ) - \bar{\chi_x}(\bar{\phi_y} \bar{\psi_z} -\bar{\phi_z}\bar{\psi_y} ) ],\\ [5pt]
 
 H_{11} &= \frac{\bar{x}\bar{y}\bar{z}}{\omega^4}\bar{\phi_y}(\bar{\psi_x}-\bar{\chi_x}) \left(\bar{y} \bar{\chi_x}\bar{\chi_y}+\bar{z}\bar{\psi_x}\bar{\psi_z}\right) +  \frac{\bar{y}\bar{z}}{\omega^2 \bar{x}\bar{\phi}_{xx}}\bar{\phi_y} \left(\bar{z} \bar{\psi_x}\bar{\chi}_{xx}-\bar{x}\bar{\chi_x}\bar{\psi}_{xx}\right), \\[5pt]
 
  \alpha(h)&= \left\{-\begin{pmatrix} (\bar{f}_1)_{xx} & (\bar{f}_1)_{xy} & (\bar{f}_1)_{xz}
\end{pmatrix} J^{-1}
 \begin{pmatrix}
(\bar{f}_1)_h  \\
(\bar{f}_2)_h\\
(\bar{f}_3)_h
 \end{pmatrix} +(\bar{f}_1)_{xh}\right\}\Big(\frac{h-\bar{h}}{\zeta}\Big) \\
 & -\frac{1}{\omega^2}\{(\bar{f}_1)_y [(\bar{f}_2)_y (\bar{f}_2)_x + (\bar{f}_2)_z (\bar{f}_3)_x] + (\bar{f}_1)_z [(\bar{f}_3)_y (\bar{f}_2)_x + (\bar{f}_3)_z (\bar{f}_3)_x) ]  \}.
      \end{array} 
\right. 
\end{eqnarray*}

 \emph{Hopf bifurcation of system (\ref{normal1}):} System (\ref{normal1}) undergoes a Hopf bifurcation at $h=\bar{h}+\zeta A+O(\zeta^{3/2})$ for sufficiently small $\zeta>0$, where $A$ is the solution of the equation
\bes \nonumber
 \left\{-\begin{pmatrix} (\bar{f}_1)_{xx} & (\bar{f}_1)_{xy} & (\bar{f}_1)_{xz}
\end{pmatrix} J^{-1}
 \begin{pmatrix}
(\bar{f}_1)_h  \\
(\bar{f}_2)_h\\
(\bar{f}_3)_h
 \end{pmatrix} +(\bar{f}_1)_{xh}\right\}A 
 & =&\frac{1}{\omega^2}\{(\bar{f}_1)_y [(\bar{f}_2)_y (\bar{f}_2)_x + (\bar{f}_2)_z (\bar{f}_3)_x] \\
  \label{appnd1} &+& (\bar{f}_1)_z [(\bar{f}_3)_y (\bar{f}_2)_x + (\bar{f}_3)_z (\bar{f}_3)_x) ]  \}.
\ees

 \emph{Transformation to normal coordinates:} The transformations from the original coordinates $(x,y,z)$ to the normal form coordinates $(u,v,w)$ are given below: 
\begin{eqnarray}  \label{convert}
\left\{
\begin{array}{ll}
 u &= \frac{\bar{f_{1xx}}}{\omega} X -\delta (\bar{f_{1y}} B_1(A_1, A_2, X, Y, Z)+\bar{f_{1z}}B_2(A_1, A_2, X, Y, Z) \\
& +  \frac{\delta}{3}C_{uv} \Big[ \Big( \frac{\bar{f_{1xx}}}{\omega} X 
-  \delta (\bar{f_{1y}} B_1(A_1, A_2, X,Y,Z) + \\
&  \bar{f_{1z}}B_2((A_1, A_2, X, Y, Z))\Big)^2 
  \Big(-1+\frac{\bar{f_{1xx}}}{2\omega^2} (\bar{f_{1y}}Y+\bar{f_1}_zZ )\Big)+ \\
& \frac{\bar{f^2_{1xx}}}{\omega^4} (\bar{f_{1y}}Y+\bar{f_{1z}}Z)^2 \Big],\\
 v &= \frac{\bar{f_{1xx}}}{\omega^2} (\bar{f_{1y}}Y+\bar{f_{1z}}Z),\\
 w &=- \frac{\bar{f_{1xx}} \bar{f_{1y}}}{\omega^2 \bar{f_{1z}}} \Big[ \Big(1+ \frac{\bar{f_{2x}}\bar{f_1}_{y}}{\omega^2}\Big) Y + {\bar{f_{2x}}\bar{f_{1z}}}{\omega^2} Z + \frac{\delta}{\omega}\bar{f_1}_{xx} A_1 X\Big],
    \end{array} 
\right. 
\end{eqnarray}
where
\bess
X = \frac{x-x_0}{\sqrt{\zeta}}, \ Y = \frac{y-y_0}{\zeta} , \ Z = \frac{z-z_0}{\zeta},
\eess
and 
\bess
A_1 & =& \left( 1+ \frac{\bar{f_{2x}}\bar{f_{1y}}}{\omega^2}\right) (\bar{f_{2y}}\bar{f_{2x}} + \bar{f_{2z}}\bar{f_{3x}}) + \frac{\bar{f_{2x}}\bar{f_{1z}}}{\omega^2} (\bar{f_{3y}}\bar{f_{2x}} + \bar{f_{3z}}\bar{f_{2x}}), \\
A_2 & = &  \frac{\bar{f_{3x}}\bar{f_{1y}}}{\omega^2} (\bar{f_{2y}}\bar{f_{2x}} + \bar{f_{2z}}\bar{f_{3x}}) + \left( 1+\frac{\bar{f_{3x}}\bar{f_{1z}}}{\omega^2}\right)  (\bar{f_{3y}}\bar{f_{2x}} + \bar{f_{3z}}\bar{f_{2x}}),\\
C_{uv} &=& -\frac{1}{\omega^2} \Big[ \bar{f_{1y}}(\bar{f_{2y}}\bar{f_{2x}} + \bar{f_{2z}}\bar{f_{3x}}) + \bar{f_{1z}} (\bar{f_{3y}}\bar{f_{2x}} + \bar{f_{3z}}\bar{f_{2x}})\Big] \\
&-& \frac{1}{\bar{f_{1xx}}} 
\Big[\bar{f_{1xy}}\bar{f_{2x}} + \bar{f_{1xz}}\bar{f_{3x}} + \bar{f_{1y}}\bar{f_{2xx}} + \bar{f_{1z}}\bar{f_{3xz}}\Big],
\\
B_1 &=& -\frac{\bar{f_{1xx}}}{\omega^4} (\bar{f_{1y}}Y+ \bar{f_{1z}}Z) (\bar{f_{2y}}\bar{f_{2x}} + \bar{f_{2z}}\bar{f_{3x}}) +\frac{\bar{f_{1xx}}\bar{f_{2xx}}}{2\omega^2} + 
\frac{\bar{f_{2y}}\bar{f_{1xx}} }{\omega^2}\Big[\left(1+ \frac{\bar{f_{2x}}\bar{f_{1y}}}{\omega^2}\right) Y \\
&+& \frac{\bar{f_{2x}}\bar{f_{1z}}}{\omega^2} Z 
+ \frac{\delta A_1 X}{\omega} \Big] +  \frac{\bar{f_{2z}}\bar{f_{1xx}} }{\omega^2}\Big[ \frac{\bar{f_{3x}}\bar{f_{1y}}}{\omega^2} Y + \Big(1+\frac{\bar{f_{3x}}\bar{f_{1z}}}{\omega^2}\Big) Z + \frac{\delta A_2 X}{\omega} \Big],\\
B_2 &=& -\frac{\bar{f_{1xx}}}{\omega^4} (\bar{f_{1y}}Y+ \bar{f_{1z}}Z) (\bar{f_{3y}}\bar{f_{2x}} + \bar{f_{3z}}\bar{f_{3x}}) +
\frac{\bar{f_{1xx}}\bar{f_{3xx}}}{2\omega^2} + 
\frac{\bar{f_{3y}}\bar{f_{1xx}} }{\omega^2}\Big[\left(1+ \frac{\bar{f_{2x}}\bar{f_{1y}}}{\omega^2}\right) Y \\
&+& \frac{\bar{f_{2x}}\bar{f_{1z}}}{\omega^2} Z 
+ \frac{\delta A_1 X}{\omega} \Big] +  \frac{\bar{f_{3z}}\bar{f_{1xx}} }{\omega^2}\Big[ \frac{\bar{f_{3x}}\bar{f_{1y}}}{\omega^2} Y + \Big(1+\frac{\bar{f_{3x}}\bar{f_{1z}}}{\omega^2}\Big) Z + \frac{\delta A_2 X}{\omega} \Big].
\eess

\vspace{0.1in}

 \emph{Proof of Lemma \ref{funnel}:}
Let $M\geq N$ be the largest integer such that $\{u(\tau_i)\}_{i=1}^M$ decreases and $w(\tau)>0$ on $[0, \tau_M]$. 
By Remark \ref{rmk1}, the expression for $\bar{u^2}_{\textrm{est}}(\tau)$, defined by (\ref{uest})-(\ref{ubase}), holds up to $O(\delta^2)$ for all $\tau \in [\tau_1, \tau_M]$. It is clear  from (P2) that $\bar{w}(\tau)$ also has a decreasing envelope for all $\tau \in [\tau_1, \tau_N]$.  Estimating $\bar{w}(\tau)$ by $\bar{w}_{\textrm{base}}(\tau)$, where  $\bar{w}_{\textrm{base}}(\tau)$ is defined by (\ref{west1}), it can be easily verified that $\bar{w}_{\textrm{base}}(\tau)$ has a unique critical point at $\tau=\bar{\tau}_m>0$, where 
\bess
\bar{\tau}_m = \frac{1}{\delta H_3-b_2}\ln \Big(\frac{H_{11}b_1b_2e^{(\delta H_3-b_2) \tau_1}}{H_3(2\bar{w}(\tau_1)(\delta H_3-b_2)+\delta H_{11}b_1)} \Big)
\eess
 if (\ref{condnew}) holds.  The critical point corresponds to a minimum of  $\bar{w}_{\textrm{base}}(\tau)$ with \[\bar{w}_{\textrm{base}}(\bar{\tau}_m)= \frac{-H_{11}b_1}{2 H_3}e^{b_2(\bar{\tau}_m-\tau_1)}>0.\] 
 Since $\bar{w}(\tau)-\bar{w}_{\textrm{base}}(\tau)=O(\delta^2)$ as long as $w(\tau) =O(\delta)$, it follows that $\bar{w}(\tau)$ and the lower envelope of $w(\tau)$ also attain their minima.  Consequently, $w(\tau)$ attains its global minimum at $\tau = {\tau}_{\mathrm{min}}$, where ${\tau}_{\mathrm{min}}= {\bar{\tau}}_m +O(\delta^2)$. Finally, by (P2), since $w(\tau)$ has a decreasing envelope on $[\tau_1, \tau_N]$, it is clear that  ${\tau}_{\mathrm{min}}\geq \tau_N$.
 
  {\hfill \ensuremath{\Box}}

 
 \emph{Proof of Lemma \ref{funnel1}:}
  We start by considering the relative position of $w(\tau)$ with respect to $-\frac{H_{11}}{2H_3} (u^2+v^2)(\tau)$ and show that there exists some $\tau_c> {\tau}_{\mathrm{min}}$  such that  $w(\tau)$ intersects with the upper envelope of $-\frac{H_{11}}{2H_3} (u^2+v^2)(\tau)$ at $\tau =\tau_c$.  By (P2), $w(\tau)<-\frac{H_{11}}{2H_3} (u^2+v^2)(\tau)$ for all  $\tau \in [0, \tau_{N}]$. 
   Let $\{t_k\}_{k=1}^{M}$ and $\{s_k\}_{k=1}^{M-1}$ be an increasing sequence of locations of relative maxima and minima of $v(\tau)$ in $I$ respectively with $t_k<\tau_k<s_k$.  Since the trajectory is spiraling inwards, we have that  $\{v(t_k)\}_{k=1}^M$ and $\{v(s_k)\}_{k=1}^{M-1}$ are decreasing and increasing respectively. Without loss of generality, we assume that $|v(t_k)|\leq |v(s_k)|$ for all $1\leq k \leq M-1$. With the aid of (\ref{normal2}), we note that the zeros of $u(\tau)$ correspond to the critical points of $(u^2 +v^2)(\tau)$. Hence the relative maxima and minima of $(u^2 +v^2)(\tau)$ occur at $s_k$ and $t_k$ respectively with values $v^2(s_k)$ and $v^2(t_k)$  for all $k$. Denoting the local maxima and minima of $-\frac{H_{11}}{2H_3} (u^2+v^2)(\tau)$ on the interval $[t_k, s_k]$ by $M_k$ and $m_k$ respectively, we then have that $\{M_k\}_{k=1}^{M-1}$ and $\{m_k\}_{k=1}^{M-1}$ are both decreasing. 
If $u_{\textrm{env}}(\tau)$ and $v_{\textrm{env}}(\tau)$ denote the upper envelopes of $u(\tau)$ and $v(\tau)$ respectively, we then have from (\ref{flow1}) that up to $O(\delta^2)$,   $u_{\textrm{env}}(\tau)= Ae^{c_2(\tau-\tau_1)}$ and $v_{\textrm{env}}(\tau)= Be^{c_2(\tau-\tau_1)}$ for $\tau \in [\tau_1, \tau_M]$, where $A$ is defined by (\ref{relA}) and
  \bes  \label{rel00}
  B= \frac{A}{c_2^2+\vartheta^2},\ 
   c_2=\frac{b_2}{2}=\frac{\alpha \delta}{2} +O(\delta^k). \ees
 Since $(u^2+v^2)(\tau)\leq  (u_{\textrm{env}}^2(\tau) +v_{\textrm{env}}^2(\tau))$ and $|w(\tau)- \bar{w}_{\textrm{base}}(\tau)|=O(\delta^2)$, it suffices to show that there exists some  $\tau_c \in ({\tau}_{\mathrm{min}}, \tau_M)$ such that $\bar{w}_{\textrm{base}}(\tau)$ intersects with  $\frac{-H_{11}}{2H_3} (u_{\textrm{env}}^2(\tau) +v_{\textrm{env}}^2(\tau))$ at $\tau = \tau_c$.  To this end, we first note from (\ref{west1}), (\ref{rel0}) and (\ref{rel00})  that $\bar{w}_{\textrm{base}}(\tau)$ intersects with  $\frac{-H_{11}}{2H_3} (u_{\textrm{env}}^2(\tau) +v_{\textrm{env}}^2(\tau))$ if and only if $\frac{-H_{11}(A^2+B^2)}{2H_3} > \frac{\delta H_{11}b_1}{2(b_2-\delta H_3)}$, i.e.
 \bes
\label{rel2} 1+\frac{1}{c_2^2+\vartheta^2} > - \frac{1-e^{\frac{-4 \pi c_2}{\vartheta}}}{\frac{8\pi  c_2}{\vartheta}\Big(\frac{2c_2}{\delta H_3}-1\Big)}.
 \ees
 Since $c_2 =O(\alpha\delta/2)$ and $\vartheta \approx 1$, it then follows that the left hand side of (\ref{rel2}) is approximately equal to $2$, whereas the right hand side of (\ref{rel2}) is less than $1/2$ for $\delta$ sufficiently small, and thus (\ref{rel2}) holds. Next, we observe from (\ref{wbar_eqn}) and (\ref{ubase}) that as long as  $\bar{w}_{\textrm{base}}(\tau)$ is decreasing, $ \bar{w}_{\textrm{base}}(\tau)\leq  \frac{-H_{11}}{2H_3} b_1$. Since $b_1 \approx A^2/2$ and $c_2<0$, we then have from the definition of $u_{\textrm{env}}$ and  $v_{\textrm{env}}$ that 
 \bess
 \bar{w}_{\textrm{base}}(\tau) \leq  \frac{-H_{11}}{2H_3} b_1< -\frac{H_{11}}{2H_3} (u_{\textrm{env}}^2(\tau) +v_{\textrm{env}}^2(\tau)) \ \textnormal{on} \ [\tau_1, \bar{\tau}_m].
 \eess
  Hence, for sufficiently small $\delta>0$,
 we must have $\tau_c > \bar{\tau}_m = \tau_{\textrm{min}} +O(\delta^2)$.


Finally, we note that the existence of $\tau_c$ implies that $w(\tau)\geq M_i$ on $[s_i, t_{i+1}]$ for some $i$ between $1$ and $N$, and hence the trajectory is in $\mathbb{R}^3 \setminus \Omega_{\alpha}$ over the interval $[s_i, t_{i+1}]$. The monotonic properties of $\{M_k\}_{k=1}^{M-1}$ and $\{m_k\}_{k=1}^{M-1}$, and the fact that $\bar{w}'(\tau)>0$ for $\tau>{\tau}_{\mathrm{min}}$ will then imply that  the trajectory is in $\mathbb{R}^3 \setminus \Omega_{\alpha}$  over the interval $[s_i, \tau_M]$. Choosing $\tau_a=s_j$, where $1<j<M$ is the smallest integer such that $s_j\leq \tau_c$ with $w(s_j)>M_j$, and $\tau_b=\tau_M$ then yields the result. 

  {\hfill \ensuremath{\Box}}

 \emph{Proof of Lemma \ref{invariance}:} Let $\{\tau_i\}_{i=1}^M$ be an increasing sequence of locations of relative maxima of $u(\tau)$ such that $\{u(\tau_i)\}_{i=1}^M$ is decreasing, where $M$ is some integer greater than $N$ to be chosen later.  By Remark \ref{rmk1}, we note that  the expressions for $\bar{u^2}_{\textrm{est}}(\tau)$ and $\bar{w}_{\textrm{base}}(\tau) $ defined by (\ref{uest}) and (\ref{west1}) respectively, hold up to $O(\delta^2)$ for all $\tau \in [\tau_1, \tau_M]$.  If $\bar{w}(\tau_1)< \frac{\delta H_{11}b_1}{2(b_2-\delta H_3)}$, then it follows from (\ref{west1}) that 
$\bar{w}_{\textrm{base}}(\tau)$ decreases and eventually becomes negative at $\tau =\tau_e$, where 
 \bess
{\tau_e} = \frac{1}{\delta H_3-b_2}\ln \Big(\frac{\delta H_{11}b_1e^{(\delta H_3 -b_2)\tau_1}}{\delta H_{11}b_2 e^{b_2 \tau_1}+2\bar{w}(\tau_1)(\delta H_3-b_2)} \Big).
\eess
 Since $ |w(\tau) - \bar{w}_{\textrm{base}}(\tau)|=O(\delta^2)$ and $\bar{w}_{\textrm{base}}'(\tau)<0$ for all $\tau \in [0, \tau_M]$, we must have that   $w(\tau)$ also changes its sign near $\tau_e$.  Let $\tilde{\tau}=\tau_e+o(1)$ be such that $w(\tilde{\tau})=0$.  Choose $M$ large enough such that $\tilde{\tau} <\tau_{M}$.
Since $q_e$ is unstable along the $w$-direction and $w'(\tilde{\tau}) =  \frac{\delta}{2} H_{11}u^2 (\tilde{\tau})\leq0$ (follows from (\ref{normal2})), we cannot have $w(\tau)=0$ for all $\tau\geq \tilde{\tau}$, thereby leading to $w(\tau)<0\leq  -\frac{H_{11}}{2H_3} (u^2+v^2)(\tau)$ for all $\tau >\tilde{\tau}$. Thus $(u(\tau), v(\tau), w(\tau))\in \Omega_{\alpha}$ for all $\tau >\tilde{\tau}$. It is also evident from  (\ref{normal2}) that  $w'(\tau) < \delta H_3 w$  and therefore if $w(\bar{\tau})=-\tilde{k}$ for some $\tilde{k}>0$ and $\bar{\tau} > \tilde{\tau}$, then by Gronwall's inequality,  $w(\tau) < -\tilde{k} e^{\delta H_3 (\tau -\bar{\tau})} \to -\infty$ as $\tau \to \infty$. Consequently, the fast variables $(u(\tau), v(\tau))$, governed by system (\ref{normal_par}), approach $(0,0)$ as $\tau \to \infty$.

Next, we will prove that $\Omega_{\alpha}$ is positively invariant with respect to 
 $(u(\tau),v(\tau),w(\tau))$ for all $\tau\geq 0$.  Denoting the sequences of local minima and maxima of  $-\frac{H_{11}}{2 H_3}(u^2+v^2)$  by $\{m_k\}_{k=1}^{M-1}$ and $\{M_k\}_{k=1}^{M-1}$ and their locations by $\{t_k\}_{k=1}^{M-1}$ and $\{s_k\}_{k=1}^{M-1}$ respectively with $t_k<\tau_k<s_k$, and recalling that $(u(\tau),  v(\tau), w(\tau)) \in \Omega_{\alpha}$  on $[0, \tau_N]$, we then have that  $w(\tau) < m_k$  on the interval $[t_k, t_{k+1})$ for all $1\leq k\leq N-1$.
As in the proof of Lemma \ref{funnel1}, we have that $m_k =-\frac{H_{11}}{2H_3}v^2(t_k) =-\frac{H_{11}B^2}{2H_3} e^{b_2 (t_k-\tau_1)}$ up to $O(\delta^2)$, where $B$ is defined by (\ref{rel00}), and thus we have 
 $\bar{w}_{\textrm{base}}(\tau)\approx w(\tau) <-\frac{H_{11}B^2}{2H_3} e^{b_2 (\tau-\tau_1)}$ on $[0, \tau_{N}]$, i.e. 
\bes \label{rel11}
\frac{\delta H_{11} b_1}{2(b_2-\delta H_3)} < -\frac{H_{11} B^2}{2H_3}.
\ees
 Since by our assumption $\bar{w}(\tau_1)< \frac{\delta H_{11}b_1}{2(b_2-\delta H_3)}$, we have from  (\ref{west1}) in combination with (\ref{rel11})  that
  $\bar{w}_{\textrm{base}}(\tau) < m_k$ on $[t_k, t_{k+1})$ for all $ 1\leq k\leq M-1$, i.e. $\bar{w}_{\textrm{base}}(\tau)< -\frac{H_{11}}{2 H_3}(u^2+v^2)(\tau)$  for all $\tau\in [0,\tau_M]$. 
Finally, since $w(\tau)=\bar{w}_{\textrm{base}}(\tau)+O(\delta^2)$, we can conclude that the solution is in $\Omega_{\alpha}$ for all $\tau\in [0,\tau_M]$. Combining this with the fact that  $(u(\tau), v(\tau), w(\tau))\in \Omega_{\alpha}$ for all $\tau >\tilde{\tau}$ proves the lemma.

  {\hfill \ensuremath{\Box}}

\section*{Acknowledgments}
The authors thank the anonymous reviewers for their suggestions that led to a significant improvement in the presentation of the article. S.S  acknowledges the Provost's Summer Research Fellowship program in Georgia College for supporting this research and is grateful for having fruitful discussions at the workshops on ``Advances in Mathematical Ecology" at the Fields Institute in December 2022 and at the University of Pittsburgh in June 2023.

\section*{Data Availability}
The manuscript has no associated data.



\begin{thebibliography}{100}



\bibitem{ashwin} P. Ashwin, S. Wieczorek, R. Vitolo and P. Cox {\em Tipping points in open systems: bifurcation, noise-induced and rate-dependent examples in the climate system}, Philos. Trans. R. Soc. Ser. A, 370 (2012), pp. 1166–1184.



\bibitem{BE} S. M. Baer and T. Erneux, {\em Singular Hopf bifurcation to relaxation oscillations}, SIAM Journal
on Applied Mathematics, 46 (1986), pp. 721-739.




\bibitem{BL} G. Bl\'e, and I. Loreto–Hern\'andez. {\em Two-dimensional attracting torus in an intraguild predator model with general functional responses and logistic growth rate for the prey}, J. Appl. Anal. Comp. 11 (3), (2021) pp. 1557–1576

\bibitem{BB}
B.~Braaksma,
\newblock {\em Singular Hopf bifurcation in systems with fast and slow variables}.
\newblock { J. Nonlinear Sci.}, 8(5), (1998) pp. 457--490.

\bibitem{biggs} R. Biggs, G. D. Peterson, and J. C. Rocha, {\em The Regime Shifts Database: a framework for analyzing regime shifts in social-ecological systems}, Ecology and Society 23 (3):9 (2018)


\bibitem{boettigera2} C. Boettigera, A. Hastings, {\em Quantifying limits to detection of early warning for critical transitions}, J.  Royal Soc.  Interface. 7; 75, (2012) pp. 2527 - 2539.

\bibitem{BKW} B.M. Br${\mathrm{\o}}$ns,  M. Krupa, M. Wechselberger, {\em Mixed Mode Oscillations Due to the Generalized Canard Phenomenon}, Fields Institute Communications 49 (2006) pp. 39-63.

\bibitem{carpenter} S.R. Carpenter and W.A.  Brock, {\em Rising variance: a leading indicator of ecological transition}. Ecol. Lett. 9 (2006) pp. 311–318


\bibitem{cantrell} S.R Cantrell and C. Cosner,  {\em Spatial Ecology via Reaction-Diffusion Equations}, Wiley, 2004.

\bibitem{gore_nature} A. Chen, A.Sanchez, L. Dai, et al. {\em Dynamics of a producer-freeloader ecosystem on the brink of collapse}, Nat Commun. 5, 3713 (2014).

\bibitem{clements} C.F. Clements, M.A. McCarthy, and J.L. Blanchard, {\em Early warning signals of recovery in complex systems}, Nat Commun 10, 1681 (2019).




\bibitem{dakos} V. Dakos, S.R. Carpenter, E. H. van Nes and M. Scheffer, {\em Resilience indicators: prospects and limitations for early warnings of regime shifts}, Philos Trans R Soc Lond B Biol Sci. (2015)  370: 20130263.

\bibitem{BD1} B. Deng, {\em Food chain chaos due to junction-fold point}, Chaos 11 (2001) pp.  514-525.

\bibitem{DH} B. Deng and G. Hines, {\em Food chain chaos due to Shilnikov's orbit}, Chaos 12 (2002)  pp. 533-538.


\bibitem{DGKKOW} M. Desroches, J. Guckenheimer, B. Krauskopf, C. Kuehn, H.M. Osinga, M. Wechselberger, {\em Mixed-Mode Oscillations with Multiple Time Scales}, SIAM Review 54 (2012) pp. 211-288.

\bibitem{Ermen} B. Ermentrout, {\em Simulating, Analyzing, and Animating Dynamical Systems: A Guide to XPPAUT for Researchers and Students},  SIAM, 2002.


\bibitem{G} J. Guckenheimer, {\em Singular Hopf Bifurcation in Systems with Two Slow Variables}, SIAM Journal of Applied Dynamical Systems 7 (2008) pp. 1355-1377.

\bibitem{GH} J. Guckenheimer and P. Holmes, {\em Nonlinear Oscillations, Dynamical Systems and Bifurcations of Vector Fields},  Springer-Verlag, Berlin 1983.



\bibitem{Hastings1}A. Hastings, {\em Transients: the key to long-term ecological understanding?}, Trends in Ecol. and Evol. 19  (2004).

\bibitem{Hastings2} A. Hastings, K. C. Abbott, K. Cuddington, T. Francis, G. Gellner, Y-C. Lai, A. Morozov, S. Petrovskii, K. Scranton, M.L. Zeeman, {\em Transient phenomena in ecology}, Science 07 (2018).

\bibitem{Hastings3} A. Hastings, K. C. Abbott, K. Cuddington, T. Francis, Y-C. Lai, A. Morozov, S. Petrovskii,  M.L. Zeeman, {\em Effects of stochasticity on the length and behaviour of ecological transients}, J R Soc. Interface.  18 (180) (2021)

\bibitem {heggerudetal} C.M. Heggerud, H. Wang, M.A. Lewis, {\em Transient dynamics of a stoichiometric cyanobacteria model via multiple-scale analysis}, SIAM Journal on Applied Math 80(3)  (2020)  pp.1223-1246.

\bibitem{kang} Y. Kang and L. Wedekin, {\em Dynamics of a intraguild predation model with generalist or specialist predator.} J. Math. Biol. 67,  (2013) pp. 1227–1259.

\bibitem{ks1}  M. Krupa and P. Szmolyan,  {\em Relaxation oscillations and canard explosion.} J Diff Equ 174,  (2001) pp. 312 - 368.

\bibitem{ks2} M. Krupa and P. Szmolyan,   {\em Extending geometric singular perturbation theory to nonhyperbolic points- 
fold and canard points in two dimensions.} SIAM J Math Anal 33 (2), (2001) pp. 286 - 314.


\bibitem{kuskeetal}  R. Kuske, L.F. Gordillo, P. Greenwood, {\em Sustained oscillations via coherence resonance in SIR},  J. Theo. Biol., 245 (3), (2007) pp. 459 - 469.



\bibitem{K} Y.A.Kuznetsov, {\em Elements of Applied BifurcationTheory}, Springer, 1998.



\bibitem{morozovetal} A. Morozov, K. Abbott, K. Cuddington, T. Francis, G. Gellner, A. Hastings, Y.C.  Lai, S. Petrovskii, K. Scranton, M.L.  Zeeman, {\em Long transients in ecology: theory and applications}, Physics of Life Reviews,  32, (2020) pp.  1 - 40. 





\bibitem{MR} S. Muratori and S. Rinaldi, {\em Remarks on competitive coexistence}, SIAM J.  Applied Math. 49 (1989)  pp. 1462-1472.

\bibitem{murray} J.D. Murray, {\em Mathematical Biology}, Springer (2002).

\bibitem{reimer} J.R. Reimer, J. Arroyo-Esquivel, J. Jiang, et al. {\em Noise can create or erase long transient dynamics}. Theor Ecol 14, 685–695 (2021)

\bibitem{MR1} S. Rinaldi  and S. Muratori, {\em Slow-fast limit cycles in predator-prey models}, Ecol.  Modelling  61 (1992) pp.  287-308.
 
 \bibitem{rubinetal} J. Rubin, D. Earn, P. Greenwood, T.L. Parsons and K. Abbott, {\em Irregular timing of population cycles}, (2021)   hal-03445622v1f.

 \bibitem{SCT} S. Sadhu and S. Chakraborty Thakur, {\em Uncertainty and Predictability in Population Dynamics of a Two-trophic Ecological Model: Mixed-mode Oscillations, Bistability and Sensitivity to Parameters}, Ecol. Complexity 32 (2017) pp. 196-208.
 

 
   \bibitem{Sadhunew} S. Sadhu, {\em Complex oscillatory patterns near singular Hopf bifurcation in a two time-scale ecosystem}, Disc. \& Cont. Dynm. Syst. B, 26 (2021), pp. 5251-5279.
   
 \bibitem{sadhutrans} S. Sadhu, {\em Analysis of the onset of a regime shift and detecting early warning signs of major population changes in a two-trophic three-species predator-prey model with long-term transients}, J. Math. Biol., 85:38 (2022).
 
 \bibitem{gore_plos} A. Sanchez, J. Gore, {\em Feedback between population and evolutionary dynamics determines the fate of social microbial populations}, PLoS Biol. (2013) 11(4):e1001547.
 
\bibitem{SC} M. Scheffer and S. R. Carpenter, {\em Catastrophic regime shifts in ecosystems: linking theory to observation}, Trends in Ecol. and Evol. 18 (2003) pp. 648 -656.


\bibitem{Scheffer} M. Scheffer, {\em Critical Transitions in Nature and Society}, 16 Princeton University Press 2009.



\bibitem{V} F. Verhulst, {\em Nonlinear Differential Equations and Dynamical Systems}, Springer-Verlag: New York, Heidelberg, Berlin, 1993.

\bibitem{ZKE} J. Zhu, R. Kuske, and T. Erneux, {\em Tipping Points Near a Delayed Saddle Node Bifurcation with Periodic Forcing}, SIAM Journal on Applied Dynamical Systems, 14, (4)  (2015) pp. 2030 - 2068.




\end{thebibliography}
\end{document}